\documentclass{amsart}

\usepackage{amsfonts}
\usepackage{amsthm}
\usepackage{amstext}
\usepackage{amsmath}
\usepackage{amscd}
\usepackage{amssymb}
\usepackage[mathscr]{eucal}
\usepackage{graphicx}
\usepackage{graphics}
\usepackage{epsf}
\usepackage{array}
\usepackage{url}

 \makeatletter
 \def\@seccntformat#1{\csname the#1\endcsname.\quad}
 \makeatother

\theoremstyle{plain}
\newtheorem{theorem}{Theorem}
\newtheorem{proposition}[theorem]{Proposition}
\newtheorem{corollary}[theorem]{Corollary}
\newtheorem{lemma}[theorem]{Lemma}

\newcounter{mainresult}
\renewcommand{\themainresult}{\Alph{mainresult}}
\newtheorem{propositionA}[mainresult]{Proposition}
\newtheorem{theoremA}[mainresult]{Theorem}
\newtheorem{corollaryA}[mainresult]{Corollary}

\theoremstyle{definition}

\newtheorem{definitionA}[mainresult]{Definition}

\theoremstyle{remark}
\newtheorem{remark}[theorem]{Remark}

\theoremstyle{remark}
\newtheorem{example}[theorem]{Example}

\numberwithin{equation}{section}

\newcommand{\mL}{L\kern-0.08cm\char39}

\newcommand{\Lip}{\operatorname{Lip}}
\newcommand{\id}{{\rm id}}
\newcommand{\End}{E}
\newcommand{\Ord}{O}
\newcommand{\Branch}{B}
\newcommand{\Cut}{\operatorname{Cut}}   

\newcommand{\NNN}{\mathbb N}
\newcommand{\ZZZ}{\mathbb Z}
\newcommand{\RRR}{\mathbb R}
\newcommand{\SSS}{\mathbb S}

\newcommand{\III}{I}

\newcommand{\CCc}{\mathcal{C}}
\newcommand{\DDd}{\mathcal{D}}

\newcommand{\HHh}{\mathcal{H}}

\newcommand{\closure}[1]{\overline{#1}}
\newcommand{\boundary}[1]{\partial {#1}}
\newcommand{\interior}[1]{\operatorname{int}(#1)}

\newcommand{\eps}{\varepsilon}
\newcommand{\ord}{\operatorname{ord}}
\newcommand{\abs}[1]{\lvert#1\rvert}
\newcommand{\card}[1]{\##1}
\newcommand{\invlim}{\varprojlim}

\newcommand{\lengthd}[2]{\HHh^1_{#1}(#2)}

\begin{document}

\title[Length-expanding Lipschitz maps on totally regular continua]
 {Length-expanding Lipschitz maps\\on totally regular continua}

\author[Vladim\'\i r \v Spitalsk\'y]{Vladim\'\i r \v Spitalsk\'y}
\address{Department of Mathematics, Faculty of Natural Sciences,
          Matej Bel University, Tajovsk\'eho 40, 974 01 Bansk\'a Bystrica,
          Slovakia}
\email{vladimir.spitalsky@umb.sk}

\subjclass[2010]{Primary 37B05, 37B20, 37B40; Secondary 54H20}

\keywords{Lipschitz map, length-expanding map, tent map,
totally regular continuum, rectifiable
curve, exact Devaney chaos, specification property.}

\begin{abstract}
The tent map is an elementary example of an interval map po\-sse\-ssing many interesting properties,
such as dense periodicity, exactness, Lip\-schitz\-ness and a kind of length-expansiveness.
It is often used in constructions of dynamical systems on the interval/trees/graphs.
The purpose of the present paper is to construct,
on totally regular
continua (i.e.~on topologically rectifiable curves),
maps sharing some typical properties with
the tent map. These maps will be called
length-expanding Lipschitz maps, briefly LEL maps.
We show that every totally regular continuum endowed with a suitable
metric admits a LEL map. As an application we obtain that every totally regular continuum
admits an exactly Devaney chaotic map with finite entropy and the specification property.
\end{abstract}

\maketitle

\thispagestyle{empty}

\section{Introduction}\label{S:intro}
The tent map is the piecewise linear map $f$ on the interval $I=[0,1]$ given by
$x\mapsto 2 \min\{x,1-x\}$. The properties of this map,
conjugate to the full logistic map $x\mapsto 4x(1-x)$,
include
Lipschitzness, length-expansiveness
(in a sense that it doubles the length of every subinterval $J$ of $I$ not containing $1/2$), exactness, specification, finite positive topological entropy and dense periodicity, just to name a few.
This map, together with ``generalized'' tent maps, i.e.~piecewise linear
continuous maps $f_k:I\to I$
($k\ge 3$) fixing $0$ and  mapping linearly every interval $[(i-1)/k,i/k]$ onto
$I$, are frequently used in dynamics.
Usefulness of these maps lies in the fact that on one hand they are
very simple (and so we have easy explicit formulae for iterates,
periodic points, horseshoes, etc.) and on the other hand they are
very ``powerful''. They are often used in constructions of systems
on the interval/trees/graphs with special properties. For example,
it is known that to construct a transitive map on the unit interval
with the smallest possible topological entropy, one can define
$g:I\to I$ in such a way that $1/2$ is a fixed point, $g$ maps
linearly $I_0=[0,1/2]$ onto $I_1=[1/2,1]$ and $g|_{I_1}:I_1\to I_0$
is ``tent-like''. Analogously one can define a transitive map with
the smallest possible entropy $(1/n)\log2$ on any $n$-star $S_n$
($n\ge 3$), see \cite{AKLS}; the map fixes the branch point of
$S_n$, maps cyclically each branch to the next one, all but one
linearly and the remaining one in a ``tent-like'' way.

Unfortunately, when one wants to construct a map with given
properties on curves more general than graphs, he/she faces the
problem that no direct analogue of the tent map on such curves is
known. Take e.g.~the $\omega$-star $X$, which is a very simple
dendrite defined as an infinite wedge of arcs. A construction of a
transitive finite entropy map on $X$ is much more complicated then
on $n$-stars and, as far as we know, no such construction has been
available in literature. The only result in this direction known to
us is the theorem of Agronsky and Ceder \cite{AC01} stating that any
finite-dimensional Peano continuum (hence also the $\omega$-star)
admits a transitive map; however, the proof does not say anything
about the entropy of the map.

The purpose of the present paper is to construct,
on continua more general than graphs, a family of maps
sharing some typical properties with the tent map.
Since the key property of these maps
will deal with the \emph{length} (Hausdorff one-dimensional measure)
of subcontinua and their images, the natural class of spaces to consider is the class
of \emph{rectifiable curves}, i.e.~continua of finite length. Topologically they
coincide with the class of totally regular continua.
Recall that a continuum $X$ is \emph{totally regular}
if for every point $x\in X$ and every countable set $P\subseteq X$
there is a basis of neighborhoods of $x$ with finite boundary not
intersecting $P$. This notion was introduced in \cite{Nik}, but the
class of these continua was studied a long time before, see e.g.~\cite{Why35,Eil38,EH,Eil44}.
For more details on totally regular continua
see Section~\ref{SS:totallyRegular}.

Before stating the main results of the paper we need to introduce
the notion of a length-expanding Lipschitz map.
Let $X$ be a non-degenerate totally regular continuum.
We say that a family $\CCc$ of
non-degenerate subcontinua of $X$
is \emph{dense} if every nonempty open set in $X$ contains a member of $\CCc$.
Recall that a map $f:(X,d)\to (X',d')$ between metric spaces is \emph{Lipschitz-$L$}
if $d'(f(x),f(y))\le L\cdot d(x,y)$ for every $x,y\in X$.
For a metric space $(X,d)$, the Hausdorff one-dimensional measure is denoted by $\HHh^1_d$.

\begin{definitionA} Let $X=(X,d)$, $X'=(X',d')$
be non-degenerate (totally regular) continua of finite length and let
$\CCc,\CCc'$ be dense systems of
subcontinua of $X,X'$, respectively.
We say that a continuous map
$f:X\to X'$ is \emph{length-expanding} with respect to $\CCc,\CCc'$
if there exists $\varrho>1$ (called \emph{length-expansivity constant} of $f$)
such that, for every $C\in \CCc$, $f(C)\in\CCc'$ and
\begin{equation}\label{EQ:defLengthExpanding}
       \text{if} \quad
       f(C)\ne X'
       \qquad\text{then}\quad
       \lengthd{d'}{f(C)} \ge \varrho\cdot \lengthd{d}{C}.
\end{equation}
Moreover, if $f$ is surjective and Lipschitz-$L$ we say
that $f:(X,d,\CCc)\to (X',d',\CCc')$ is \emph{$(\varrho,L)$-length-expanding Lipschitz}.
Sometimes we briefly say that $f$ is
\emph{$(\varrho,L)$-LEL} or only \emph{LEL}.
On the other hand, when we wish to be more precise, we say that $f$ is \emph{$(\CCc,\CCc',\varrho,L)$-LEL}.
\end{definitionA}

A few comments are necessary. Assume that $f:(X,d,\CCc)\to
(X',d',\CCc')$ is $(\varrho,L)$-LEL and denote by $\CCc_X$ and
$\CCc_{X'}$ the systems of \emph{all} subcontinua of $X$ and $X'$,
respectively. Obviously, then also $f:(X,d,\CCc)\to
(X',d',\CCc_{X'})$ is {$(\varrho,L)$-LEL. However, one cannot claim
that $f:(X,d,\CCc_X)\to (X',d',\CCc')$ is $(\varrho,L)$-LEL. In
fact, for some spaces $(X,d)$, $(X',d')$ there is no LEL map
$f:(X,d,\CCc_X)\to (X',d',\CCc_{X'})$. For instance this is the case
when $X$ is the $\omega$-star  and $X'=I$. To show this, suppose
that there is a $(\varrho,L)$-LEL map $f:(X,d,\CCc_X)\to
(X',d',\CCc_{X'})$. Take $k\in\NNN$ such that $\varrho>L/k$ and find
a $k$-star $C$ in $X$ such that every edge of $C$ is mapped onto the
same proper subinterval of $X'$. Then $\lengthd{d'}{f(C)}\le
(L/k)\cdot\lengthd{d}{C}<\varrho\cdot\lengthd{d}{C}$, a
contradiction.

Our first result says that in the special case when $X=X'$ and
$\CCc=\CCc'$, LEL maps have interesting dynamical properties. (For
the definitions of the corresponding notions, see
Section~\ref{S:preliminaries}.)

\newcommand{\propExact}{
Let $f:(X,d,\CCc)\to (X,d,\CCc)$ be a LEL map. Then $f$ is exact and
has finite positive entropy. Moreover, if $f$ is the composition
$\varphi\circ\psi$ of some maps $\psi:X\to I$ and $\varphi:I\to X$,
then $f$ has the specification property and so   it is exactly
Devaney chaotic. }

\begin{propositionA}\label{P:tentLikeIsExact}
\propExact{}
\end{propositionA}

The above mentioned tent-like maps $f_k:I\to I$ (where $k\ge 3$ and
$I$ is endowed with the Euclidean metric $d_I$) are
$(\CCc_I,\CCc_I,k/2,k)$-LEL, where $\CCc_I$ is the system of all
non-degenerate closed subintervals of $I$. Here $k\ge 3$ because the
classical tent map $f_2$ is not $(\CCc_I,\CCc_I,\varrho,L)$-LEL for
any $\varrho>1$ and any $L$. However, it becomes
$(\CCc_I,\CCc_I,\varrho,L)$-LEL (for some $\varrho>1$ and $L$) after
a slight change of the metric.
One can easily construct examples of LEL maps between arbitrary graphs, even in
the form of the composition $\varphi\circ\psi$ as in Proposition~\ref{P:tentLikeIsExact};
one can use e.g.~the maps from \cite[Lemma~3.6]{ARR}.
Further, for a given continuum
$(X,d)$ of finite length, one can often find $\CCc,\CCc'$
and construct LEL-maps $\varphi:(I,d_I,\CCc_I)\to (X,d,\CCc)$ and
$\psi:(X,d,\CCc')\to (I,d_I,\CCc_I)$.
However, it is not so easy to obtain $\CCc'\supseteq \CCc$; this inclusion is
desirable since then also the composition $\psi\circ \varphi$
is LEL (see Lemma~\ref{L:tentLikeComposition}).

Our main results, the proofs of which were inspired by \cite{AC01}
and \cite{BNT}, assert that such LEL maps can always be found
provided we allow to change the metric on $X$ (the new metric still
being compatible with the topology). Recall that a metric $d$ on $X$
is \emph{convex} if for every $x,y\in X$ there is $z\in X$ such that
$d(x,z)=d(z,y)=d(x,y)/2$. For two points $a,b\in X$ of a continuum
$X$, $\Cut_X(a,b)$ denotes the set of points $x\in X$ such that
$a,b$ lie in different components of $X\setminus\{x\}$.

\newcommand{\thmMainA}{
 For every non-degenerate totally regular continuum $X$ and every $a,b\in X$ we can find a convex metric
 $d=d_{X,a,b}$ on $X$ and Lipschitz surjections
 $\varphi_{X,a,b}:I\to X$, $\psi_{X,a,b}:X\to I$
 with the following properties:
\begin{enumerate}
  \item[(a)] $\lengthd{d}{X}= 1$;
  \item[(b)] the system
    $\CCc=\CCc_{X,a,b}=\{\varphi_{X,a,b}(J):\ J\text{ is a closed subinterval of } I\}$ 
    is a dense system of subcontinua of $X$;
  \item[(c)] for every $\varrho>1$ there are
  a constant $L_\varrho$ (depending only on $\varrho$) and
  $(\varrho,L_\varrho)$-LEL maps
   $$
    \varphi:(I,d_I,\CCc_I)\to (X,d,\CCc)
    \quad\text{and}\quad
    \psi:(X,d,\CCc)\to (I,d_I,\CCc_I)
   $$
    with $\varphi(0)=a$, $\varphi(1)=b$ , $\psi(a)=0$ and such that
    $\varphi=\varphi_{X,a,b}\circ f_k$, $\psi=f_l\circ \psi_{X,a,b}$ for some
    $k,l\ge 3$.
\end{enumerate}
Moreover, if $\Cut_X(a,b)$ is uncountable, $d,\varphi,\psi$ can be assumed to satisfy
\begin{enumerate}
  \item[(d)] $d(a,b)>1/2$ and $\psi(b)=1$.
\end{enumerate}
}

\begin{theoremA}\label{T:MAINa}
\thmMainA{}
\end{theoremA}

\newcommand{\thmMain}{
Keeping the notation from Theorem~\ref{T:MAINa},
for every $\varrho>1$, every non-degenerate totally regular continua $X,X'$ and every points
$a,b\in X$, $a',b'\in X'$
there are a constant $L_\varrho$ (depending only on $\varrho$) and $(\varrho,L_\varrho)$-LEL map
$$f:(X,d_{X,a,b},\CCc_{X,a,b})\to (X',d_{X',a',b'},\CCc_{X',a',b'})$$
with $f(a)=a'$ and, provided $\Cut_X(a,b)$ is uncountable,
$f(b)=b'$. Moreover, $f$ can be chosen to be the composition
$\varphi\circ\psi$ of two LEL-maps $\psi:X\to I$ and $\varphi:I\to
X'$. }

\begin{theoremA}\label{T:MAIN}
\thmMain{}
\end{theoremA}

In \cite{AC01} it was shown that every non-degenerate finite-dimensional Peano continuum admits
an exactly Devaney chaotic map and that every finite union of non-degenerate
finite-dimensional Peano continua admits a Devaney chaotic map.
Theorem ~\ref{T:MAIN}
and Proposition~\ref{P:tentLikeIsExact} imply the following results which, on one hand,
deal with smaller class of spaces, but on the other hand ensure finiteness of the entropy.

\newcommand{\corCont}{Every non-degenerate totally regular continuum admits an exactly Devaney
  chaotic map with finite positive entropy and specification.}
\newcommand{\corUnion}{Every finite union of non-degenerate totally regular continua admits a Devaney
  chaotic map with finite positive entropy.}

\begin{corollaryA}\label{C:mainCont}
\corCont{}
\end{corollaryA}

\begin{corollaryA}\label{C:mainUnion}
\corUnion{}
\end{corollaryA}

In a subsequent paper we deal with the problem of determining the infima of entropies
of transitive/exact/(exactly) Devaney chaotic maps on a given totally regular continuum
and we show that under some conditions this infimum is zero. The constructions are
heavily based on Theorems~\ref{T:MAINa} and \ref{T:MAIN}.
To illustrate usefulness of LEL maps
let us sketch here an example
which shows how easy is to construct a small entropy transitive system on the $\omega$-star.

\begin{example}
Let $X$ be the $\omega$-star with the branch point $a$ and edges $A_i$ ($i=1,2\dots$);
i.e.~$X=\bigcup_i A_i$ and $A_i\cap A_j=\{a\}$ for every $i\ne j$.
Take arbitrarily large $k$, put $Y=\bigcup_{i\ge k} A_i$ and define a convex metric $d$
on $X$ in such a way that it coincides with $d_{Y,a,a}$
on $Y$ and each of the sets $A_1,\dots,A_{k-1}$ has length $1$. 
Fix $\varrho>1$. By Theorem~\ref{T:MAINa} there are $(\varrho,L_\varrho)$-maps
$f_{k-1}:A_{k-1}\to Y$, $f_k:Y\to A_1$ fixing $a$. Let
$f_i:A_i\to A_{i+1}$ ($i=1,\dots,k-2$) be isometries fixing $a$. Then it suffices to define
$f:X\to X$ by $f|_{A_i}=f_i$ for $i<k$ and $f|_Y=f_k$. The map $f^k|_Y:Y\to Y$ is exact and
has dense periodic points by Proposition~\ref{P:tentLikeIsExact}; 
moreover, it is Lipschitz-$L_\varrho^2$. So $f$ is Devaney chaotic with entropy
$h(f)\le (2/k)\log L_\varrho$, where $L_\varrho$ does not depend on $k$.
\end{example}

The paper is organized as follows. In the next section we give
an outline of the proofs of Theorems~\ref{T:MAINa} and \ref{T:MAIN}.
In Section~\ref{S:preliminaries}
we recall all the needed definitions and facts.
In Section~\ref{S:tentLike} we prove some basic properties of LEL maps.
The main
part of the paper
--- Sections~\ref{S:proofOfMainThms} and \ref{S:varphi} --- are devoted to the
construction of LEL maps from the unit interval onto a
given totally regular continuum and vice versa,
see Proposition~\ref{T:main}. Finally, in
Section~\ref{S:applications} we prove the main results of the paper,
namely Theorems~\ref{T:MAINa}, \ref{T:MAIN} and
Corollaries~\ref{C:mainCont}, \ref{C:mainUnion}.

\section{Outline of the proofs of Theorems~\ref{T:MAINa} and \ref{T:MAIN}}\label{S:mainThm}

Since the proofs of Theorems~\ref{T:MAINa} and \ref{T:MAIN} consist of a series of lemmas
and propositions, for reader's
convenience we
decided to summarize here the main steps of them. To increase
readability we skip some technical details, hence the outline
is only ``informal'' view of the proofs.

In Section~\ref{S:proofOfMainThms}, for a given totally regular continuum $X$,
 we construct a convex metric $d$ and two
 Lipschitz-1 surjections $g:[0,\alpha]\to (X,d)$
 and $h:(X,d)\to [0,\beta]$ such that $\lengthd{d}{X}\le 1$ and
  \begin{equation*}\label{EQ:phi-psi-properties}
   \gamma\cdot \abs{J} \le \lengthd{d}{g(J)}\le \Gamma\cdot\abs{h\circ g(J)}
  \end{equation*}
  for every closed subinterval $J$ of $[0,\alpha]$, where
  $0<\gamma<\Gamma$ are constants not depending on $J$; see Lemma~\ref{P:props-g-h}.
  The metric $d$ and maps $g,h$ are defined as follows.

  \begin{itemize}
    \item By \cite{BNT} we can realize $X$ as the inverse limit
     $$
       X = \invlim (X_n,f_{n})
     $$
     of graphs $X_n$ with monotone surjective bonding maps $f_n:X_{n+1}\to X_n$
    ($n=1,2,\dots$), see (\ref{EQ:invlim}).
    We may assume that for every $n$ there is exactly one point $\tilde{x}_n$ of $X_n$
    having non-degenerate $f_{n}$-preimage
    $\tilde{X}_{n+1}=f^{-1}_{n}(\tilde{x}_n)$.
  \item The convex metric $d$ is defined by
   $$
         d(x,y)
         =\sup\limits_{n\in\NNN} d_n(x_n,y_n)
         \qquad\text{for }
         x=(x_n)_n,\ y=(y_n)_n\in X,
   $$
   see (\ref{EQ:def-d}). Here the metric $d_1$ on $X_1$ is defined by Lemma~\ref{L:graph:admissible2}
   and (\ref{EQ:metric-d1}), and the metrics
   $d_n$ on $X_n$ ($n\ge 2$) are defined inductively in such a way that $d_n$
   ``coincides'' with $d_{n-1}$ on $X_n\setminus \tilde{X}_{n}$ and the length of
   $\tilde{X}_{n}$ is ``very small'' when compared to the length of any
   edge of $X_{n-1}$, see (\ref{EQ:mu-n})--(\ref{EQ:def-dn}).

  \item In Lemmas~\ref{L:metric-d}--\ref{L:d-H1} we prove that
        $d$ is a convex metric on $X$ compatible with the topology
        and that $\lengthd{d}{X}\le 1$.

  \item In (\ref{EQ:g-def}) we define $g:[0,\alpha]\to X$
   as the inverse-limit map
   $$
    g = \invlim g_n,
   $$
   where $g_n:[0,\alpha_n]\to X_n$ are
   natural parametrizations of appropriately chosen paths in $X_n$; see
   (\ref{EQ:def-In})--(\ref{EQ:cd-gnk}).

  \item The map $h:X\to[0,\beta]$ is defined in (\ref{EQ:def-h}) simply by
   $$
    h(x) = d(a,x) \qquad\text{for }x\in X,
   $$
   where $a\in X$ is a point fixed in advance.

  \item In Lemma~\ref{P:props-g-h} we summarize the properties of $d$, $g$ and $h$.
  \end{itemize}

\begin{figure}[ht!]
  \includegraphics{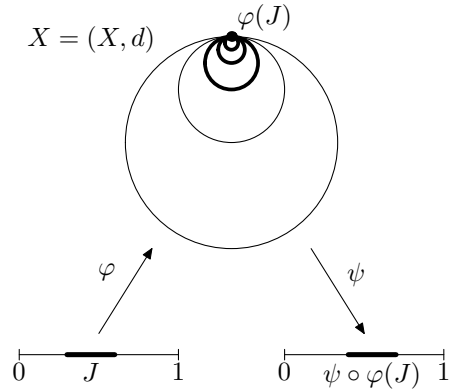}
  \caption{The maps $\varphi$ and $\psi$}
  \label{Fig:thmA}
\end{figure}

In Section~\ref{S:varphi} we construct LEL maps $\varphi:I\to X$
and $\psi:X\to I$ basically by
 linear reparametrizations of $g,h$, see Proposition~\ref{T:main},
 Corollary~\ref{C:tentLikeInterval}
  and, for an illustration, Figure~\ref{Fig:thmA}.

Using the above described tools and results, Theorems~\ref{T:MAINa} and \ref{T:MAIN}
can already be easily proved (the proofs themselves can be found in 
Section~\ref{S:applications}).

\section{Preliminaries}\label{S:preliminaries}
Here we briefly recall all the notions and results which will be needed
in the rest of the paper. The terminology is taken mainly from \cite{Kur2, Nad, Macias, Fal}.

If $M$ is a set, its cardinality is denoted by $\card M$.
The cardinality of infinite countable sets is denoted
by $\aleph_0$. If $M$ is a singleton set we often identify it with its only point.
We write $\NNN$ for the set of positive integers $\{1,2,3,\dots\}$,
$\RRR$ for the set of reals and $\III$ for
the unit interval $[0,1]$.
By an interval we mean any nonempty connected subset of $\RRR$
(possibly degenerate to a point).
For intervals
$J,J'$ we write $J\le J'$ if $t\le s$ for every $t\in J$, $s\in J'$.

By a \emph{space} we mean any nonempty metric space. A space is
called \emph{degenerate} provided it has only one point; otherwise
it is called \emph{non-degenerate}. If $E$ is a subset of a space
$X=(X,d)$ we denote the closure, the interior and the boundary of
$E$ by $\closure{E}$, $\interior{E}$ and $\boundary{E}$,
respectively, and we write $d(E)$ for the diameter of $E$. We say
that two sets $E,F\subseteq X$ are \emph{non-overlapping} if they
have disjoint interiors. For $x\in X$ and $r>0$ we denote the closed
ball with the center $x$ and radius $r$ by $B(x,r)$.
If $f$ is a map defined on $X$ and $\CCc$ is a system of subsets of $X$
we denote the system $\{f(C):\ C\in\CCc\}$ by $f(\CCc)$.

A \emph{(discrete) dynamical system} is a pair $(X,f)$ where $X=(X,d)$ is
a compact metric space and $f:X\to X$ is a continuous map.
For
$n\in\NNN$ we denote the composition $f\circ f\circ \dots\circ f$
($n$-times) by $f^n$.
A point $x\in X$ is a \emph{periodic point} of
$f$ if $f^n(x)=x$ for some $n\in\NNN$.
The \emph{topological entropy} of a dynamical system $(X,f)$ is denoted by $h(f)$.
We say that $(X,f)$ is \emph{(topologically) transitive} if for every
nonempty open sets $U,V\subseteq X$ there is $n\in\NNN$ such that
$f^n(U)\cap V\ne\emptyset$. A system $(X,f)$ is
\emph{(topologically) exact} or \emph{locally eventually onto} if
for every nonempty open subset $U$ of $X$ there is $n\in\NNN$ such
that $f^n(U)=X$. Further, $(X,f)$ is \emph{Devaney chaotic}
(\emph{exactly Devaney chaotic})
provided $X$ is infinite, $f$ is transitive (exact) and has dense set of periodic points.
Finally, a system $(X,f)$ is said to satisfy the \emph{specification property}
if for every $\eps>0$ there is $m$ such that for every $k\ge  2$, for every $k$ points $x_1,\dots,x_k\in X$, for every integers $a_1\le b_1<\dots<a_k\le b_k$ with
$a_i-b_{i-1}\ge m$ ($i=2,\dots,k$) and for every integer $p\ge m+b_k-a_1$, there
is a point $x\in X$ with $f^p(x)=x$ such that
$$
 d(f^n(x),f^n(x_i))\le \eps
 \qquad\text{for}\quad
 a_i\le n\le b_i,\ 1\le i\le k.
$$

\subsection{Continua}\label{SS:continua}
A \emph{continuum} is a connected compact metric space.
A \emph{cut point} (or a \emph{separating point}) of a continuum $X$ is any point $x\in X$ such that
$X\setminus\{x\}$ is disconnected. A point $x$ of a continuum $X$
is called a \emph{local separating point} of $X$ if there is a connected
neighborhood $U$ of $x$ such that $U\setminus\{x\}$ is not connected.
If $a, b$ are points of $X$ then
any cut point of $X$ such that $a,b$ belong to different components of
$X\setminus\{x\}$ is said to \emph{separate $a,b$}. The set of all
such points is denoted by $\Cut(a,b)$ or $\Cut_X(a,b)$. If $a=b$ then obviously $\Cut(a,b)=\emptyset$.

Let $X$ be a continuum, let $x\in X$ and let $m$ be a cardinal number.
We say that the \emph{order} of $x$ is at most $m$, written $\ord_X(x)\le m$,
provided $X$ has a local basis of open neighborhoods of $X$ the boundary of which
has cardinality at most $m$. If $m$ is the least such cardinal we write $\ord_X(x)=m$
with one exception: if $m=\aleph_0$ and
$x$ has a basis of neighborhoods with finite boundary, we write
$\ord_X(x)=\omega$.
If $\ord_X(x)=\omega$ or $\ord_X(x)$ is finite we say that $x$ has \emph{finite order}
and we write $\ord_X(x)\le \omega$.
The points of order $1$ are called \emph{end points}, the points of
order $2$ are called \emph{ordinary points} and the points of order at least $3$ are called
\emph{branch points} of $X$; the sets of all end, ordinary and branch points are
denoted by $\End(X)$, $\Ord(X)$ and $\Branch(X)$, respectively.

Tightly connected with the order of a point is the following notion, see e.g.~\cite{Why35}.
A point $x$ of a continuum $X$ is said to be of \emph{degree} $m$, written $\deg_X(x)=m$,
provided $m$ is the least cardinal such that for every $\eps>0$
there exists an \emph{uncountable} family of neighborhoods of $x$ with diameters less than $\eps$,
each having the boundary of cardinality at most $m$ and such that for any two neighborhoods
$U,V$ either $\closure{U}\subseteq V$ or $\closure{V}\subseteq U$. Again if $m=\aleph_0$
and the neighborhoods can be chosen with finite boundary we write $\deg_X(x)=\omega$
instead of $\deg_X(x)=\aleph_0$.
We say that $x$ has \emph{finite degree} and write $\deg_X(x)\le\omega$
if the degree of $x$ is either finite or $\omega$.
Trivially always $\ord_X(x)\le \deg_X(x)$ but there are examples when $\ord_X(x)< \deg_X(x)$;
e.g.~if $X$ is  the Sierpi\'nski triangle then the order of every point $x\in X$ is at most $4$ and
the degree is equal to the cardinality of the continuum \cite{Why35}.

Let $X$ be a continuum. A metric $d$ on $X$ is said to be \emph{convex}
provided for every distinct $x,y\in X$
there is $z\in X$ such that $d(x,z)=d(z,y)=d(x,y)/2$.
By \cite[Theorem~8]{BingPartSet} every locally connected
continuum admits a compatible convex metric.

\subsection{Graphs}\label{SS:graphs}
An \emph{arc} $A$ in $X$ is any homeomorphic image of the unit interval $\III$;
the end points of $A$ are the images of the points $0,1$. A \emph{simple closed curve}
is any homeomorphic image of the unit circle $\SSS^1$.

By a \emph{graph} we mean a continuum which can be written as the union
of finitely many arcs which are either disjoint or intersect only at their end
points. These arcs are called \emph{edges} and their end points are called
\emph{vertices} of the graph. So we allow vertices of order $2$ and thus
the edges and vertices are not defined uniquely.
Notice also that we do not allow simple closed curves to be edges of a graph.
By a \emph{subgraph} of a graph $G$ we mean any non-degenerate subcontinuum $H$ of $G$;
so the vertices/edges of $H$ need not be vertices/edges of $G$.

Let $G=(G,d)$ be a graph with $\lengthd{d}{G}<\infty$
and let $a,b$ be vertices of $G$.
By a \emph{path} in $G$ from $a$ to $b$ we mean a sequence
$\pi=a_0 E_1 a_1 E_2 \dots a_{k-1} E_k a_k$, where $a_i$ ($i=0,\dots,k$) are vertices of $G$ such that
$a_0=a$, $a_k=b$ and $E_j$ ($j=1,\dots,k$) are edges of $G$
with end points $a_{j-1},a_j$; the number $k$ will be called
the \emph{length} of the path $\pi$. A \emph{natural parametrization}
of a path $\pi=a_0 E_1 a_1 E_2 \dots a_{k-1} E_k a_k$
is any continuous map $\kappa:J\to G$ defined on a compact interval $J=[s,t]\subseteq \RRR$
such that $\kappa(s)=a_0$, $\kappa(t)=a_k$ and we can write
$J$ as the union $J_1\cup J_2\cup\dots\cup J_k$ of non-overlapping
closed subintervals such that $J_1\le J_2\le \dots\le J_k$ and
the restriction of $\kappa|_{J_j}:J_j\to E_j$ is an isometry for every $j=1,\dots,k$.

\subsection{Hausdorff one-dimensional measure and Lipschitz maps}\label{SS:hausdorffMeasure}
For a Borel subset $B$ of a metric
space $(X,d)$ the \emph{one-dimensional Hausdorff measure} of $B$ is defined by
$$
 \lengthd{d}{B}=\lim_{\delta\to 0} \lengthd{d,\delta}{B},
 \qquad
 \lengthd{d,\delta}{B} = \inf\left\{
   \sum_{i=1}^\infty d(E_i):\ B\subseteq \bigcup_{i=1}^\infty E_i,\ d(E_i)<\delta
 \right\}.
$$
We say that $(X,d)$ has \emph{finite length} if $\lengthd{d}{X}<\infty$.
By e.g.~\cite[Proposition~4A]{Fre92},
\begin{equation}\label{EQ:length-diam}
\lengthd{d}{C}\ge d(C)
\qquad
\text{whenever }
C
\text{ is a connected Borel subset of }
X.
\end{equation}
If $A\subseteq X$ is an arc then $\lengthd{d}{A}$ is equal to the length of $A$
\cite[Lemma~3.2]{Fal}.
In the case when $(X,d)$ is the Euclidean real line $\RRR$
and $J\subset \RRR$ is an interval $\lengthd{d}{J}$
is equal to the length of $J$ and we denote it simply by $\abs{J}$.

If $(X,d)$ is a continuum of finite length endowed with a convex metric $d$,
then it has the
so-called \emph{geodesic property} (see e.g.~\cite[Corollary~4E]{Fre92}):
for every distinct $x,y\in X$
there is an arc $A$ with end points $x,y$ such that $d(x,y)=\lengthd{d}{A}$;
any such arc $A$ is called a \emph{geodesic arc} or shortly a \emph{geodesic}.
Every subarc of a geodesic is again a geodesic. If $x,y$ are the end points
of a geodesic $A$ and $z\in A$ then $d(x,y)=d(x,z)+d(z,y)$.


A map $f:(X,d)\to (Y,\varrho)$ between metric spaces is called
\emph{Lipschitz} with a Lipschitz constant $L\ge 0$, shortly
\emph{Lipschitz-$L$}, provided $\varrho(f(x),f(x')) \le L\cdot
d(x,x')$ for every $x,x'\in X$; the smallest such $L$ is denoted by
$\Lip(f)$ and is called the \emph{Lipschitz constant} of $f$. If
$f:X\to Y$ is Lipschitz-$L$ then $\lengthd{\varrho}{f(B)} \le L\cdot \lengthd{d}{B}$
for every Borel set $B\subset X$ such that $f(B)$ is Borel-measurable
\cite[p.~10]{Fal}.
We omit the proof of the following lemma.

\begin{lemma}\label{L:CX} Let $X=(X,d)$ be a non-degenerate (totally regular) continuum
of finite length and let $\varphi:I\to X$ be a Lipschitz surjection.
Put
$$
 \CCc=\varphi(\CCc_I) = \{\varphi(J):\ J\text{ is a non-degenerate closed subinterval of } J\}.
$$
Then the following hold:
\begin{enumerate}
    \item[(1)] $\CCc$ is a dense system of subcontinua of $X$;
    \item[(2)] for every $\eps>0$ the space $X$ can be covered by some $C_1,\dots,C_k\in\CCc$
     satisfying $\lengthd{d}{C_i}<\eps$ for $i=1,\dots,k$.
\end{enumerate}
\end{lemma}

\subsection{Totally regular continua}\label{SS:totallyRegular}
By e.g.~\cite{Kur2, Nik}, a
continuum $X$ is called
\begin{itemize}
    \item a \emph{dendrite} if it is locally connected and contains no simple closed curve;
    \item a \emph{local dendrite} if it is locally connected and contains at most finitely many
     simple closed curves;
    \item \emph{completely regular} if it contains no
     non-degenerate nowhere dense subcontinuum;
    \item \emph{totally regular} if for every $x\in X$ and every countable set $P\subseteq X$
     there is a basis of neighborhoods of $x$ with finite boundary not intersecting $P$;
    \item \emph{regular} if every $x\in X$ has a basis of neighborhoods with
     finite boundary, i.e.~$\ord_X(x)\le \omega$ for every $x$;
    \item \emph{hereditarily locally connected} if every subcontinuum of $X$ is locally
     connected;
    \item \emph{rational} if every $x\in X$ has a basis of neighborhoods with
     countable boundary, i.e.~$\ord_X(x)\le \aleph_0$ for every $x$;
    \item a \emph{curve} if it is one-dimensional.
\end{itemize}
Notice that (local) dendrites as well as completely regular continua
are totally regular and (totally) regular continua are hereditarily locally connected, hence they are
locally connected curves.
Totally regular continua are also called \emph{continua of finite degree}
since
they are just those continua $X$ for which every point $x$
has finite degree $\deg_X(x)\le \omega$ \cite{Eil38}.
This and other conditions equivalent to total regularity
are summarized in the following theorem.

\begin{theorem}\label{T:totallyRegularEquivalentConditions}
For a continuum $X$ the following are equivalent:
\begin{enumerate}
    \item[(1)] $X$ is totally regular;
    \item[(2)] $X$ is of finite degree (i.e.~$\deg_X(x)\le\omega$ for every $x$);
    \item[(3)] $X$ has a (convex) metric $d$ such that $(X,d)$ has finite length;
    \item[(4)] $X$ has a (convex) metric $d$ such that $(X,d)$ is a Lipschitz image
     of the unit interval;
    \item[(5)] $X$ has a (convex) metric $d$ such that
     for every $x\in X$ and for almost every $r>0$ the boundary of the closed ball
     $B(x,r)$ is finite;
    \item[(6)] every non-degenerate
      subcontinuum of $X$ contains uncountably many local separating points;
    \item[(7)] $X$ is locally connected and for every disjoint closed sets $E,F\subseteq X$
     there are disjoint perfect sets $N_1,\dots,N_k$ such that
     every subcontinuum of $X$ intersecting both $E$ and $F$ contains some $N_i$.
\end{enumerate}
\end{theorem}
\begin{proof}
 The equivalence of (1), (2), (3), (6) and (7) follows from
 \cite{Eil38}, \cite{Why35}, \cite{EH}, \cite{Har44} and \cite{Eil44}. Immediately
 (4) implies (3) and (5) implies (2).
 By e.g.~\cite[Lemma~2A]{Fre94}, (3) implies (4).
 Finally, the fact that (3) implies (5) follows from the following inequality
 (see e.g.~\cite[1A(f)]{Fre92} or \cite[Theorem~7.7]{Mattila})
 applied to the map $f:X\to\RRR$, $f(x')=d(x,x')$.
 The inequality says that if $f:(X,d)\to (Y,\varrho)$ is
 Lipschitz-$1$ then
 $$
  \lengthd{d}{X} \ge
  \int_Y^* \card^* f^{-1}(y)\, d\lengthd{\varrho}{y}
 $$
 where $\card^*$ denotes the cardinality of a finite set and $\infty$ for infinite sets
 and $\int_Y^* h\,d\mu$ is the infimum of integrals $\int_Y g\,d\mu$ as
 $g$ runs over $\mu$-measurable functions from $Y$ to $[0,\infty]$ such that
 $g\ge h$.
 Hence if $(X,d)$ has finite length then $f^{-1}(y)$ is finite
 for $\HHh_\varrho^1$--almost every $y\in Y$.
\end{proof}

By \cite{Eil44}, if $d$ is a metric on a totally regular continuum $X$ with
$\lengthd{d}{X}<\infty$ then there is a (unique) convex metric $d^*$ on $X$ such that
$d^*(x,y)\ge d(x,y)$ for every $x,y\in X$
and $\lengthd{d^*}{B}=\lengthd{d}{B}$ for every  Borel $B\subseteq X$; it is defined
by
$
 d^*(x,y)=\inf\left\{
  \lengthd{d}{A}:\ A \text{ is an arc from } x \text{ to } y
 \right\}
$.

\subsection{Monotone inverse limits}\label{SS:invlim}
An \emph{inverse sequence} is a sequence $(X_n,f_n)_{n\in\NNN}$ where
$X_n$ is a compact metric space and $f_n:X_{n+1}\to X_n$ is a continuous map
for every $n\in\NNN$. The \emph{inverse limit} of an inverse sequence
$(X_n,f_n)_{n\in\NNN}$ is the subspace
$X_{\infty}=\invlim (X_n,f_n)$ of the product
$\prod\limits_{n=1}^\infty X_n$ given by
$$
 X_{\infty}
 = \invlim (X_n,f_n)
 = \left\{
  (x_n)_{n=1}^\infty \in \prod\limits_{n=1}^\infty X_n:\ f_n(x_{n+1})=x_n
  \text{ for every } n\in\NNN
 \right\}.
$$
The maps $f_n$ are called \emph{bonding maps}. For $n\in\NNN$ the
projection from $X_\infty$ onto the $n$-th coordinate will be denoted by $\pi_n:X_\infty\to X_n$.
From now on we will assume that every $f_n$ (and hence every $\pi_n$) is surjective.

A fundamental result states that the inverse limit of continua is a continuum
\cite[Theorem~2.1]{Nad}. Moreover, if the dimension of every $X_n$
is at most $d$ then also $\dim X_\infty\le d$ \cite[Theorem~1.13.4]{EngDim}.
Hence the inverse limit of curves is a curve.

The special case important for us is when the bonding maps are monotone.
(Recall that a continuous map $f:X\to Y$ is \emph{monotone} if every preimage $f^{-1}(y)$ is connected.) Then also every projection map $\pi_n$ is monotone
\cite[Proposition~2.1.13]{Macias}. The following theorem combines
\cite[Corollary~2.1.14]{Macias}, \cite[Theorem~3.6]{Nik} and \cite[Theorem~10.36]{Nad}.

\begin{theorem}\label{T:monotoneInvLim1}
 Let $X_\infty=\invlim (X_n,f_n)$ be the inverse limit of continua $X_n$
 with surjective monotone bonding maps.
 If every $X_n$ is locally connected (totally regular, a dendrite)
 then also $X_\infty$ is locally connected (totally regular, a dendrite).
\end{theorem}

It is often the case that a continuum $X$ is homeomorphic to the inverse limit of some ``simpler'' continua $X_n$.
For example every continuum is the inverse limit of compact connected polyhedra
\cite[Theorem~2.15]{Nad} and every curve is the inverse limit of graphs \cite[Theorem~1.13.2]{EngDim}.
Fundamental results for \emph{monotone} inverse limits and locally connected curves are summarized
below, see \cite[Theorem~2.2]{Nik}, \cite[Theorem~3]{BNT} and e.g.~the proof of
\cite[Theorem~10.32]{Nad}.

\begin{theorem}\label{T:monotoneInvLim2}
Every locally connected curve (totally regular continuum, dendrite)
is the monotone inverse limit of regular continua (graphs, trees).
\end{theorem}

Notice that for non-degenerate totally regular continua and for non-degenerate dendrites
the bonding maps
$f_n$ in the previous theorem can be chosen such that
$f_n^{-1}(x)$ is non-degenerate for exactly one point $x$.

\medskip

The following theorem gives us a way to define the so-called \emph{induced map}
between inverse limits, see e.g.~\cite[Theorems~2.1.46--48]{Macias}.

\begin{theorem}\label{T:inducedMap}
 Let $(X_n,f_n)_n$, $(X_n',f_n')_n$ be inverse sequences and let
 $g_n:X_n\to X_n'$ ($n\in\NNN$) be continuous maps such that
 for every $n$ the left-hand side diagram commutes:
\begin{center}

\begin{tabular}{ccc}

$
  \begin{CD}
   X_{n+1} @>{f_n}>> X_n \\
   @V{g_{n+1}}VV @VV{g_n}V \\
   X_{n+1}' @>{f_n'}>> X_n'
  \end{CD}
$
& \qquad\qquad\qquad &

$
  \begin{CD}
   X_{\infty} @>{\pi_n}>> X_n \\
   @V{g_{\infty}}VV @VV{g_n}V \\
   X_{\infty}' @>{\pi_n'}>> X_n'
  \end{CD}
$
\end{tabular}
\end{center}
Then there is a unique continuous map $g_\infty=\invlim g_n: X_\infty \to X_\infty'$
such that
for every $n$ the right-hand side diagram commutes.
The map $g_\infty$ is given by
$$
 g_\infty(x_1,x_2,x_3,\dots) = (g_1(x_1), g_2(x_2), g_3(x_3), \dots).
$$
Moreover, if every $g_n$ is surjective (injective) then $g_\infty$ is surjective (injective).
\end{theorem}


\section{Properties of length-expanding Lipschitz maps}\label{S:tentLike}

Here we briefly state basic properties of the class of LEL maps.
We start with the proof of Proposition~\ref{P:tentLikeIsExact} stated in the introduction.

\renewcommand{\themainresult}{\ref{P:tentLikeIsExact}}
\begin{propositionA}
\propExact{}
\end{propositionA}
\begin{proof}
 Let $f:(X,d,\CCc)\to (X,d,\CCc)$ be a $(\varrho,L)$-LEL map.
 Take any nonempty open subset $U$ of $X$ and fix some
 $C\in\CCc$ contained in $U$.
 Then $f^n(C)\in\CCc$ for every $n$.
 If $f^n(C)\ne X$ for every $n$ then $\lengthd{d}{f^n(C)}\ge \varrho^n\cdot \lengthd{d}{C} \to \infty$ for $n\to\infty$, which contradicts the fact that $X=(X,d)$ has finite length.
 So $f^n(U)\supseteq f^n(C)=X$ for some $n$, which proves the exactness of $f$.

 Now assume that $f=\varphi\circ\psi$.
 Since $f$ is exact, also the factor $f'=\psi\circ\varphi:I\to I$
 of $f$ is exact. Hence $f'$ has the specification property by \cite{Blo83}.
 By \cite[21.4]{DGS} also $f$, being a factor of $f'$, has the specification property.
 Finally, by \cite[21.3]{DGS}, $f$ has dense periodic points.
\end{proof}

Recall that $d_I$ denotes the Euclidean metric
on $I$ and $\CCc_I$ is the system of all non-degenerate
closed subintervals of $I$.
Note that the following lemma can be substantially generalized, but for our purposes
this version is sufficient.

\begin{lemma}\label{L:tentLike_fk}
Let $k\ge 3$ and $f_k:I\to I$ be the piecewise linear map fixing $0$ and
mapping every $[(i-1)/k,i/k]$ onto $I$. Then $f_k:(I,d_I,\CCc_I)\to (I,d_I,\CCc_I)$ is $(k/2,k)$-LEL.
\end{lemma}
\begin{proof}
Only length-expansiveness needs a proof. Take any non-degenerate closed subinterval $J$ of $I$.
If there is $i$ such that $J\supseteq [(i-1)/k,i/k]$ then $f(J)=I$. Otherwise
there is $i$ such that $J=J_0\cup J_1$ where $J_0\subseteq ((i-1)/k,i/k]$ and
$J_1\subseteq [i/k,(i+1)/k)$. Then $\abs{f_k(J)} \ge k\cdot \max\{\abs{J_0},\abs{J_1}\}
\ge (k/2)\cdot \abs{J}$.
\end{proof}

\begin{lemma}\label{L:CCc'}
Let $f:(X,d,\CCc)\to (X',d',\CCc')$ be a $(\varrho,L)$-LEL map.
Then for every $\DDd'\supseteq \CCc'$, $1<\varrho'\le\varrho$ and
$L'\ge L$, the map
$f:(X,d,\CCc)\to (X',d',\DDd')$ is $(\varrho',L')$-LEL.
\end{lemma}

\begin{lemma}\label{L:tentLikeComposition}
Let $f:(X,d,\CCc)\to (X',d',\CCc')$ be $(\varrho,L)$-LEL
and $f':(X',d',\CCc')\to (X'',d'',\CCc'')$ be $(\varrho',L')$-LEL.
Then $f'\circ f:(X,d,\CCc)\to (X'',d'',\CCc'')$ is $(\varrho\varrho',LL')$-LEL.
\end{lemma}
\begin{proof}
 Put $g=f'\circ f$. Immediately $\Lip(g)\le LL'$. Take any $C\in\CCc$
 and put $C'=f(C)\in\CCc'$. If $C'=X'$ then, by surjectivity of $f'$, $g(C)=X''$.
 Otherwise $\lengthd{d'}{C'}\ge\varrho\lengthd{d}{C}$ and, if $f'(C')\ne X''$,
 also $\lengthd{d''}{f'(C')}\ge\varrho'\lengthd{d'}{C'}$; hence
 $\lengthd{d''}{g(C)}\ge\varrho\varrho'\lengthd{d}{C}$.
\end{proof}


\section{Lipschitz-$1$ surjections $g:[0,\alpha]\to X$, $h:X\to [0,\beta]$}\label{S:proofOfMainThms}

In this section we show that for a totally regular continuum $X$
there are a compatible convex metric $d$ and two
Lipschitz surjections $g:[0,\alpha]\to (X,d)$, $h:(X,d)\to [0,\beta]$ such that
$$
 \gamma\cdot \abs{J} \le \lengthd{d}{g(J)}\le \Gamma\cdot\abs{h\circ g(J)}
$$
for every closed subinterval $J$ of $[0,\alpha]$, where
$0<\gamma<\Gamma$ are constants not depending on $J$ (see Lemma~\ref{P:props-g-h}).

We start with a simple property of convex metrics on locally connected continua.
For a metric space $X=(X,d)$ and a point $a\in X$ put
\begin{equation}\label{EQ:def-ha}
 h_a:X\to \RRR,
 \qquad
 h_a(x)=d(a,x)
 \quad\text{for } x\in X.
\end{equation}

\begin{lemma}\label{L:freeArc}
Let $X=(X,d)$ be a locally connected continuum endowed with a convex metric $d$
and let $a\in X$. Then
$$
 \abs{h_a(A)}\ge \frac{1}{2} \cdot \lengthd{d}{A}
$$
for  any free arc $A$ in $X$.
\end{lemma}
\begin{proof}
Let $y,z$ be the end points of $A$. For distinct $u,v\in A$ we will denote by $uv$ the subarc of $A$ with end points $u,v$. Let $\alpha$ be the length of $A$ and let
$\kappa:[0,\alpha]\to A$ be the natural parametrization of $A$ such that $\kappa(0)=y$ and $\kappa(\alpha)=z$.
Put $y_t=\kappa(t)$ for $t\in[0,\alpha]$; hence $\lengthd{d}{y_t y_s}=\abs{s-t}$ for every
different $t,s\in[0,\alpha]$.

For every $t\in [0,\alpha]$ such that $y_t\ne a$ take a geodesic arc $A_t$ from $a$ to $y_t$.
Assume first that $a$ is not an interior point of $A$.
Since $A$ is a free arc, every arc (hence also every $A_t$) from $a$ to a point of $A$ must contain $y$ or $z$.
Take any $t\in [0,\alpha]$.
If $y\in A_t$
then $d(a,y_t)=d(a,y)+d(y,y_t)=d(a,y)+t$ since $A_t$ is geodesic.
Analogously, if $z\in A_t$ then
$d(a,y_t)=d(a,z)+d(z,y_t)=d(a,z)+(\alpha-t)$.
So
$$
 h_a(y_t)=\min\{ d(a,y)+t, d(a,z)+(\alpha-t) \}.
$$
Hence immediately $\abs{h_a(A)}\ge \alpha/2$.

Now assume that $a=y_s$ for some $s\in (0,\alpha)$;
without loss of generality we may assume that $\lengthd{d}{ay}\le \lengthd{d}{az}$. Then for every
$t\in[0,\alpha]$, $t\ne s$ the geodesic arc $A_t$ is either the subarc $ay_t$ of $A$
or an arc containing both $y$ and $z$. Hence
$$
 h_a(y_t)=\min\{
  \abs{t-s},
  s  + d(y,z) + (\alpha-t)
 \}.
$$
So also in this case we easily have $\abs{h_a(A)}\ge \alpha/2$.
\end{proof}

\subsection{Admissible maps on graphs}\label{SS:graph-case}

Let $G$ be a graph with a metric $d$ and let $a,b$ be
(not necessarily distinct) vertices of $G$.
We say that a path $\pi=a E_{j_1} a_1 \dots a_{k-1} E_{j_k} b$ in $G$
(from $a$ to $b$) is \emph{admissible} provided
every edge of $G$ is at least once but at most twice in $\pi$; moreover, if
$a_i$ is a vertex of $G$ of order $2$ then $E_{j_{i}}\ne E_{j_{i+1}}$
(i.e.~$\pi$ ``goes through'' the ordinary vertices of $G$).

A continuous map $\kappa$ from a compact interval $J=[\alpha,\beta]$ to $G$ is called
\emph{fully-admissible} or, more precisely, \emph{fully-admissible for $(G,d)$ from $a$ to $b$},
if it is the natural parametrization of some admissible path $\pi$ from $a$ to $b$.
I.e.~$\kappa(\alpha)=a$, $\kappa(\beta)=b$
and there is an admissible path $\pi_\kappa=a E_{j_1} a_1 \dots a_{k-1} E_{j_k} b$
and non-overlapping compact intervals $J_1\le J_2\le \dots \le J_k$
such that $J=J_1\cup \dots \cup J_k$ and
every restriction $\kappa|_{J_i}:J_i\to E_{j_i}$ is an isometry.

A map is called \emph{admissible} if it is a restriction of a
fully-admissible map onto a compact interval. Notice that any
admissible map is finite-to-one and outside of a finite set (the set of
points mapped to the vertices of $G$) is at most two-to-one.
Moreover, admissible maps are Lipschitz-$1$ provided the metric $d$
is convex. The following lemma can be easily proved by
induction on the number of edges of $G$.

\begin{lemma}\label{L:graph:admissible} Let $G$ be a graph and let $a,b$ be vertices of $G$.
Then there is a fully-admissible map $\kappa:[\alpha,\beta]\to G$
for $G$ from $a$ to $b$.
\end{lemma}

\begin{lemma}\label{L:graph:admissible2}
Let $0<q<1$ and let $G$ be a graph.
Then there is a convex metric $d$ on $G$ such that
for every admissible map $\kappa:J\to G$ and every vertex $a$ of $G$ it holds that
$$
 \lengthd{d}{\kappa(J)}\ge \frac{1}{2}\cdot \abs{J}
 \qquad\text{and}\qquad
 \abs{h_a\circ \kappa(J)}\ge \frac{1-q}{6}\cdot \lengthd{d}{\kappa(J)}.
$$
Moreover, $\abs{h_a(G)}\ge \dfrac{1-q}{2}\cdot \lengthd{d}{G}$.
\end{lemma}

\begin{proof}
Fix any $0<q<1$.
Let $G$ be a graph and let $E_0,\dots,E_k$ be the edges of $G$.
Take a convex metric $d$ on $G$ such that $\lengthd{d}{G}<\infty$ and
\begin{equation}\label{EQ:P:propertyX:graphs:1}
 \lengthd{d}{E_i}\le q\cdot \lengthd{d}{E_{i-1}}
 \qquad
 \text{for every } i\ge 1.
\end{equation}
Such a metric can be constructed as follows: We may assume that $G$
is a subset of $\RRR^3$ endowed with the Euclidean metric and that
the (Euclidean) lengths of edges of $G$ are finite and exponentially
decreasing with quotient $q$. Then it suffices to take the convex
metric on $G$ generated by the Euclidean one.

Let $a$ be a vertex of $G$ and let $\kappa:J\to G$ be an admissible
map for $(G,d)$; put $Y=\kappa(J)$. Let $\pi=a_0 E_{i_1} a_1
E_{i_2} \dots a_{k-1} E_{i_k} a_k$ be the admissible path given by a
fully-admissible extension of $\kappa$. Since $\pi$ is admissible,
we immediately have $\lengthd{d}{Y}\ge \frac{1}{2} \cdot \abs{J}$.

Now we show the lower bound for
the length of $h_a(Y)$. Realize that there are at most two indices $j$ such that
\begin{equation}\label{EQ:P:propertyX:graphs:3}
 Y\cap E_j \text{ is non-degenerate and } Y\not\supseteq E_j
\end{equation}
(indeed, for any such $j$ the edge $E_j$ must contain the $\kappa$-image of an end point of $J$ in its interior).
For simplicity we will assume that there are exactly two $j$'s satisfying (\ref{EQ:P:propertyX:graphs:3}) --- we denote them by $j_1,j_2$ --- and that there is an index $j$ such that $E_j\subseteq Y$; the other cases can be described analogously.
Let $j_0$ be the smallest index $j$ such that $E_j\subseteq Y$. Then using (\ref{EQ:P:propertyX:graphs:1}) we have
\begin{equation}\label{EQ:P:propertyX:graphs:4}
\begin{split}
 \lengthd{d}{Y} &= \sum_j \lengthd{d}{E_j\cap Y}
 \ \le\
 \lengthd{d}{E_{j1}\cap Y} + \lengthd{d}{E_{j_2}\cap Y} + \sum_{j\ge j_0} \lengthd{d}{E_j}
 \ \le \
\\
 &\le \
 \lengthd{d}{E_{j_1}\cap Y} + \lengthd{d}{E_{j_2}\cap Y} + \lengthd{d}{E_{j_0}}/(1-q).
\end{split}
\end{equation}
On the other hand, Lemma~\ref{L:freeArc} gives
\begin{equation*}
\begin{split}
 \abs{h_a(Y)}
 \ &\ge \
 \max\{
  \abs{h_a(E_{j_1}\cap Y)}, \abs{h_a(E_{j_2}\cap Y)}, \abs{h_a(E_{j_0})}
 \}
\\
 \ &\ge \
 \frac 12\cdot
 \max\{
  \lengthd{d}{E_{j_1}\cap Y}, \lengthd{d}{E_{j_2}\cap Y}, \lengthd{d}{E_{j_0}}
 \}.
\end{split}
\end{equation*}
The simple fact that
\begin{equation}\label{EQ:max-sum}
 \max\limits_{i=1,\dots,p} c_i
 \le
 \sum\limits_{i=1}^{p} c_i
 \le
 p\cdot  \max\limits_{i=1,\dots,p} c_i
 \qquad
 \text{for any non-negative }c_1,\dots,c_p,
\end{equation}
applied to (\ref{EQ:P:propertyX:graphs:4})
immediately implies
$\abs{h_a(Y)}\ge   \lengthd{d}{Y}\cdot (1-q)/6$.

The final assertion of the lemma follows from the facts that
$\lengthd{d}{G}\le \lengthd{d}{E_0}/(1-q)$ and $\abs{h_a(G)}\ge
\abs{h_a(E_0)} \ge \lengthd{d}{E_0}/2$ by Lemma~\ref{L:freeArc}.
\end{proof}

\subsection{The construction of $d,g,h$}\label{SS:construction}
Now we embark on the construction of a convex metric $d$ on $X$ and
Lipschitz surjections
$g:[0,\alpha]\to X$, $h:X\to [0,\beta]$
for a given totally regular continuum $X$ (see Lemma~\ref{P:props-g-h}).

Let $0<q<1$, let $X$ be a non-degenerate totally regular continuum and let
$a,b$ be two points of $X$.
By \cite{BNT} there is an inverse sequence
$(X_n,f_{n})_{n\in\NNN}$ of graphs $X_n$ with monotone surjective bonding maps $f_{n}:X_{n+1}\to X_n$
such that $X$ is (homeomorphic to) the inverse limit
\begin{equation}\label{EQ:invlim}
 \lim_{\longleftarrow} (X_n,f_{n}) \, .
\end{equation}
Without loss of generality we may assume that for every integer $n\ge 1$ the following hold:
\begin{itemize}
    \item there is $\tilde{x}_n\in X_n$ such that
      $\tilde{X}_{n+1}=f_{n}^{-1}(\tilde{x}_n)$
      is a non-degenerate subgraph of $X_{n+1}$;
    \item $f_{n}^{-1}(x)$ is a singleton for every $x\ne \tilde{x}_n$;

  \medskip

    \item $\tilde{x}_n$ is a vertex of $X_n$;
    \item every vertex of $\tilde{X}_{n+1}$ is a vertex of $X_{n+1}$;
      moreover, every point of the boundary
      (in $X_{n+1}$) of $\tilde{X}_{n+1}$ is a vertex of
      both $\tilde{X}_{n+1}$ and
      $X_{n+1}$; so an edge of $\tilde{X}_{n+1}$ is also an edge of ${X}_{n+1}$;
    \item the $f_{n}$-preimage of every vertex $x\ne \tilde{x}_n$ of $X_n$
     is a vertex of $X_{n+1}$; so the $f_{n}$-image of any edge in $X_{n+1}$ which is not
     an edge of $\tilde{X}_{n+1}$ is a free arc contained in an edge of $X_n$.
\end{itemize}
Let $\pi_n:X\to X_n$ ($n\in\NNN$) be the natural projections; put
$a_n=\pi_n(a)$, $b_n=\pi_n(b)$.
We may assume that $a_n,b_n$ are
vertices of $X_n$ and, if $a\ne b$,
$a_1\ne b_1$
(otherwise we remove finitely many of the first $X_n$'s). Then
$a_n\ne b_n$ for every $n$
provided $a\ne b$.

Let $d_1$ be a convex metric on $X_1$ obtained using Lemma~\ref{L:graph:admissible2}
such that
\begin{equation}\label{EQ:metric-d1}
 \lengthd{d_1}{X_1}= 1-q
\end{equation}
and let $g_1:I_1\to X_1$ be a fully-admissible map for $(X_1,d_1)$ from $a_1$ to $b_1$.
Assume that $n\ge 2$ and that for every $1\le m\le n-1$ we have defined
a metric $d_m$ on $X_m$,
maps $g_m:I_m\to X_m$ and, provided $m\ge 2$, a map $\varrho_{m-1}: I_{m}\to I_{m-1}$.
Put
\begin{equation}\label{EQ:mu-n}
 \mu_{n-1}= \min\{\lengthd{d_{n-1}}{E}:\ E \text{ is an edge of } X_{n-1}\}.
\end{equation}
Let $\tilde{d}_n$ be a convex metric on $\tilde{X}_n$ obtained from Lemma~\ref{L:graph:admissible2} such that
\begin{equation}\label{EQ:def-dn-tilde}
 \lengthd{\tilde{d}_n}{\tilde{X}_n} < \frac{q\cdot \mu_{n-1}}{2p}
 \qquad\text{where}\quad
  p=\card g_{n-1}^{-1}(\tilde{x}_{n-1}).
\end{equation}
Denote by $d_n$ the only convex metric on $X_n$ such that for every edge $E$ of $X_n$
and every two points $x,y\in E$ the following holds:
\begin{equation}\label{EQ:def-dn}
 d_n(x,y) =
 \begin{cases}
  \tilde{d}_n(x,y)                                      
  & \text{if } E\subseteq \tilde{X}_n;
 \\
  d_{n-1}(f_{n-1}(x), f_{n-1}(y)) & \text{otherwise}.
 \end{cases}
\end{equation}

\bigskip

Let $s_1<s_2<\dots < s_p$ be the points of $I_{n-1}$
mapped by $g_{n-1}$ to $\tilde{x}_{n-1}$. Write
$I_{n-1}$ as the union $J_0'\cup J_1'\cup\dots\cup J_p'$
of non-overlapping compact subintervals such that
$J_0'\le s_1\le J_1'\le s_2 \dots\le s_p\le J_p'$
(here $J_0',J_p'$ can be degenerate).
For every $i=1,\dots,p$ let $K_i'$ be an interval and $\kappa_i:K_i'\to (\tilde{X}_n,\tilde{d}_n)$
be a fully-admissible map (see Lemma~\ref{L:graph:admissible}); the images of end points
of $K_i'$ will be fixed later.
Now let $I_n=[0,\alpha_n]$ be a compact interval of length $\alpha_n=\abs{I_{n-1}} + \sum_{i=1}^p \abs{K_i'}$
and define $g_n:I_n\to X_n$ by ``concatenating'' the maps
$$
 g_{n-1}|_{J_0'},\
 \kappa_1,\
 g_{n-1}|_{J_1'},\
 \kappa_2,\
 \dots,
 \kappa_p,\
 g_{n-1}|_{J_p'}.
$$
I.e.~we write $I_n$ as the union of non-overlapping compact intervals
\begin{equation}\label{EQ:def-In}
 I_n = J_0 \cup K_1 \cup J_1 \cup K_2 \dots \cup K_p \cup J_p
\end{equation}
such that $J_0\le K_1\le\dots\le K_p\le J_p$ and $\abs{J_i}=\abs{J_i'}$,
$\abs{K_j}=\abs{K_j'}$ for every $i,j$; then we define $g_n$ such that
\begin{equation}\label{EQ:def-gn}
 g_n|_{J_i} \approx g_{n-1}|_{J_i'}
 \quad\text{and}\quad
 g_n|_{K_j} \approx \kappa_j
 \qquad\text{for every }i,j.
\end{equation}
(Here we write $f\approx g$ for maps $f,g$ defined on real intervals $J,K$ if
there is a constant $s_0$ such that $J=K+s_0$ and $f(s+s_0)=g(s)$ for every $s\in K$.)
By an ``appropriate'' specification of $\kappa_i$-images of the end points of $K_i'$ we obtain that $g_n$ is continuous and that $g_n(0)=a_n$, $g_n(\alpha_n)=b_n$. Notice that
\begin{equation}\label{EQ:gn-is-parametrization}
\begin{split}
 g_n:I_n\to (X_n,d_n)
  \quad\text{is a natural parametrization of some }
 \\
  \text{(not necessarily admissible) path in } X_n \text{ from } a_n \text{ to } b_n.\quad
\end{split}
\end{equation}

Let $\varrho_{n-1}:I_n\to I_{n-1}$ be the piecewise linear continuous surjection
with slopes $0$ and $1$ which collapses every $K_i$ into a point. For $1\le k<n$ denote the composition
$\varrho_{k}\circ\varrho_{k+1}\circ\dots\circ\varrho_{n-1}$ by
$\varrho_{n,k}:I_n\to I_k$; for
convenience put $\varrho_{n,n}=\id_{I_n}$. Analogously define
$f_{n,k}:X_n\to X_k$ for $1\le k\le n$. Notice that the following diagram commutes
for every $1\le k\le n$:
\begin{equation}\label{EQ:cd-gnk}
      \begin{CD}
       X_{n}  @>{f_{n,k}}>> X_{k}
      \\
       @A{g_{n}}AA @AA{g_{k}}A
      \\
       I_{n}  @>{\varrho_{n,k}}>> I_{k}
      \end{CD}
\end{equation}

\bigskip

After finishing the induction we obtain the metrics $d_n$ on $X_n$ and
the maps $g_n:I_n\to X_n$.
As in \cite{BNT} define
\begin{equation}\label{EQ:def-d}
 d(x,y)
 =\sup\limits_{n\in\NNN} d_n(x_n,y_n)
 \qquad\text{for }
 x=(x_n)_n,\ y=(y_n)_n\in X.
\end{equation}
(In Lemma~\ref{L:metric-d} we will show that $d$ is a convex metric on $X$.)
Define also
$$
 I_\infty = \invlim (I_n,\varrho_n).
$$
The corresponding projection map from $I_\infty$ onto $I_n$  ($n\in\NNN$)
will be denoted by $\pi_n'$. It is easy to see that the
map
$$
 \eta:I_\infty \to [0,\alpha],
 \quad
 (s_n)_{n\in\NNN} \mapsto t
  =\lim\limits_{n\to\infty} s_n  =\sup\limits_{n\in\NNN} s_n,
 \qquad\text{where }
 \alpha=\lim_{n\to\infty} \alpha_n,
$$
defines a homeomorphism of $I_\infty$ onto the interval $[0,\alpha]$,
which is even isometry
if we use the following metric $d'$ on $I_\infty$ (see Lemma~\ref{L:Iinfty-props}):
\begin{equation}\label{EQ:d'-def}
 d'(s,t) = \sup_n \abs{s_n-t_n}
 \qquad
 \text{for }s=(s_n)_n, t=(t_n)_n\in I_\infty.
\end{equation}

Since the diagrams in (\ref{EQ:cd-gnk}) commute,
the surjective maps $g_n:I_n\to X_n$ induce the continuous surjective
map $g=\invlim\{g_n\}:I_\infty\to X$ between
 $I_\infty = \invlim(I_n,\varrho_n)$ and
$X = \invlim(X_n,f_n)$ such that the following diagram commutes
\begin{equation}\label{EQ:cd-g}
      \begin{CD}
       X  @>{\pi_{n}}>> X_{n}
      \\
       @A{g}AA @AA{g_{n}}A
      \\
       I_{\infty}  @>{\pi_{n}'}>> I_{n}
      \end{CD}
\end{equation}
(see Theorem~\ref{T:inducedMap}); the map $g$ is given by
\begin{equation}\label{EQ:g-def}
 g(s_1,s_2,s_3\dots)=(g_1(s_1),g_2(s_2),g_3(s_3),\dots).
\end{equation}

Finally define $h_n:X_n\to \RRR$, $h:X\to\RRR$ by
\begin{equation}\label{EQ:def-h}
 h_n(x_n)=d_n(x_n,a_n) \quad\text{for } x_n\in X_n,
 \qquad
 h(x)=d(x,a)  \quad\text{for } x\in X.
\end{equation}

\subsection{Properties of the metrics $d_n$, $d$}\label{SS:proof-props-dn}
Notice that (\ref{EQ:def-dn-tilde}) and the fact that every
$\tilde{X}_n$ is non-degenerate immediately give
\begin{equation}\label{EQ:mun-q}
 \mu_{n} < q\cdot \mu_{n-1}
 \qquad\text{for every }n\ge 2.
\end{equation}
Since $\mu_1 \le 1-q$ by (\ref{EQ:metric-d1}) we have
\begin{equation}\label{EQ:mun-q2}
 \mu_n
 \le q^{n-1}\cdot (1-q)
 \qquad\text{for every }n\in\NNN.
\end{equation}

\begin{lemma}\label{L:metric-d}
 The map $d$ is a convex metric on $X$ compatible with the topology of $X$.
\end{lemma}
\begin{proof} (See \cite{BNT}.)
Let $n\ge 2$. From the definition (\ref{EQ:def-dn}) of the metrics $d_n$ we have that
for every $x,y\in X_n$, $x'=f_{n-1}(x)$, $y'=f_{n-1}(y)$
\begin{equation}\label{EQ:dn}
 d_{n-1}(x',y')
 \le d_n(x,y)
 < d_{n-1}(x',y') + q\cdot \mu_{n-1}
\end{equation}
and that, for every free arc $A$ in $X_n\setminus \interior{\tilde{X}_n}$,
\begin{equation}\label{EQ:fn-isometry}
 f_{n-1}|_A: A\to f_{n-1}(A)
 \quad\text{is a bijection and}\quad
 \lengthd{d_n}{A}=\lengthd{d_{n-1}}{f_{n-1}(A)}.
\end{equation}

Combining (\ref{EQ:dn}) and (\ref{EQ:mun-q}) we obtain
that for $m>n$
$$
 d_m(x_m,y_m)
 < d_n(x_n,y_n) + q\cdot(\mu_{m-1}+\dots +\mu_{n})
 < d_n(x_n,y_n) + \frac{q}{1-q}\cdot \mu_n,
$$
so, since $\mu_n\le q^{n-1}\mu_1 \le q^{n-1}(1-q)$,
\begin{equation}\label{EQ:d-dn}
 d_n(x_n,y_n)
 \le
 d(x,y)
 \le d_n(x_n,y_n) + q^{n}
 \qquad
 \text{for every } n\in\NNN,\ x,y\in X.
\end{equation}
Hence $d(x,y)$ is always finite. Since trivially
$d$ is symmetric, satisfies the triangle inequality and $d(x,y)=0$ if and only if
$x=y$, $d$ defines a metric on $X$.

To prove that the metric $d$ is compatible with the topology of $X$ it suffices to show
that any sequence $(x^{(k)})_k$ converges to $x$ in $(X,d)$ if and only if
$(\pi_n x^{(k)})_k$ converges to $\pi_n x$ in $(X_n,d_n)$ for every $n$.
The implication from the left to the right is trivial. Assume that $\lim_k d_n(x_n^{(k)},x_n)= 0$
for every $n$, where $x_n^{(k)}=\pi_n x^{(k)}$ and $x_n=\pi_n x$.
By (\ref{EQ:d-dn}) we have for every $n$
$$
 \limsup_{k\to\infty} d(x^{(k)},x)
 \le
 \limsup_{k\to\infty} \left[  d_n(x_n^{(k)},x_n) + q^n  \right]
 =q^n.
$$
Since $q<1$ we have that $\limsup_{k} d(x^{(k)},x)=0$.

Now it suffices to show that $d$ is convex.
Let $x=(x_n)_n$, $y=(y_n)_n\in X$. For every $n$ the metric $d_n$ is convex
so there is a point $z_n\in X_n$ such that
$$
 d_n(x_n,z_n)=d_n(y_n,z_n)=\frac 12\cdot d(x_n,y_n).
$$
Let $z^{(n)}\in X$ be such that $\pi_n(z^{(n)})=z_n$. Compactness of
$X$ gives that $(z^{(n)})_n$ has a subsequence $(z^{(n_k)})_k$ converging
to a point $z\in X$. Now (\ref{EQ:d-dn}) gives
\begin{equation*}
\begin{split}
 d(x,z)
 &\le  d(x,z^{(n_k)}) +  d(z^{(n_k)},z)
 \le  d_{n_k}(x_{n_k},z_{n_k}) + q^{n_k} + d(z^{(n_k)},z)
\\
 &=    \frac{1}{2}\cdot d_{n_k}(x_{n_k},y_{n_k})   + q^{n_k} + d(z^{(n_k)},z).
\end{split}
\end{equation*}
Using (\ref{EQ:d-dn}) and taking the limit $k\to \infty$ and  we obtain
$d(x,z)\le \frac{1}{2}\cdot d(x,y)$.
Analogously, $d(y,z)\le \frac{1}{2}\cdot d(x,y)$ and so
$d(x,z)=d(y,z)=\frac{1}{2}\cdot d(x,y)$. Hence $d$ is convex.
\end{proof}

The inequalities (\ref{EQ:d-dn}) immediately imply
that for the diameters of a subset $B$ of $X$ and
its projections $B_n=\pi_n(B)$ ($n\in \NNN$) it holds that
\begin{equation}\label{EQ:d-dn-set}
 d_n(B_n)
 \le
 d(B)
 \le d_n(B_n) + q^{n}.
\end{equation}
In Lemma~\ref{L:d-H1} we show a relation between the Hausdorff one-dimensional
measure of a subset of $X$ and of its projections.
To this end we need the following refinement of (\ref{EQ:d-dn-set}) for the special case when
$B$ is a subcontinuum of $X$.

\begin{lemma}\label{L:diam-Y-Yn}
 Let $Y$ be a subcontinuum of $X$ and let $n\in\NNN$.
 Put $Y_n=\pi_n(Y)$ and for every integer $k\ge n$ put
 $\tilde{y}_k=f_{k,n}(\tilde{x}_k)\in X_n$.
 Then
 $$
  d(Y)
  \le d_n(Y_n)
   + q\cdot \sum\limits_{k\ge n,\ \tilde{y}_{k}\in Y_n} \mu_k.
 $$
\end{lemma}
\begin{proof}
 For $m>n$ let $N_{m,n}$ be the set of all integers $k\in\{n,n+1,\dots,m-1\}$ such that
 $\tilde{y}_{k}\in Y_n$. By the definition (\ref{EQ:def-d}) of $d$ it suffices to show
 that for every $m>n$
 \begin{equation}\label{EQ:L:diam-Y-Yn:1}
  d_m(Y_m)
  \le d_n(Y_n)
   + q\cdot \sum\limits_{k\in N_{m,n}} \mu_k.
 \end{equation}
 We prove this by induction through $m-n$. Assume first that $m-n=1$; then $N_{m,n}$ is
 either the singleton $\{n\}$ or an empty set  according to whether
 $\tilde{y}_n=\tilde{x}_n$ either belongs to $Y_n$ or not.
 In the latter case
 $\tilde{y}_n=\tilde{x}_n\not\in Y_n$; since $Y_{n+1}$ is a subgraph of $X_{n+1}$,
 (\ref{EQ:fn-isometry}) and convexness of $d_{n+1}$  give that
 $d_{n+1}(Y_{n+1})=d_n(Y_n)$. In the former case (when $\tilde{x}_n\in Y_n$)
 we analogously have
 $d_{n+1}(Y_{n+1})\le d_n(Y_n) + d_{n+1}(\tilde{X}_{n+1})
 \le d_n(Y_n) + q\cdot \mu_n$. Hence (\ref{EQ:L:diam-Y-Yn:1})
 is true for any $m,n$ such that $m-n=1$.

 Now assume that for some $p\ge 1$ (\ref{EQ:L:diam-Y-Yn:1}) is true whenever $m-n\le p$;
 let $m,n$ be such that $m-n=p+1$. By the induction hypothesis
 \begin{eqnarray*}
  d_m(Y_m)
  &\le& d_{n+1}(Y_{n+1})
   + q\cdot \sum\limits_{k\in N_{m,n+1}} \mu_k
 \\
  d_{n+1}(Y_{n+1})
  &\le& d_{n}(Y_{n})
   + q\cdot \sum\limits_{k\in N_{n+1,n}} \mu_k.
 \end{eqnarray*}
 Now (\ref{EQ:L:diam-Y-Yn:1}) follows since
 $N_{m,n} = N_{m,n+1}\cup N_{n+1,n}$
 (indeed, for $k\ge n+1$ we have $f_{k,n+1}(\tilde{x}_k)\in Y_{n+1}$ if and only if
 $f_{k,n}(\tilde{x}_k)\in Y_{n}$).
\end{proof}

\begin{lemma}\label{L:d-H1}
 Let $B\subseteq X$ be a closed set and let $B_n=\pi_n(B)\subseteq X_n$  for every $n\in \NNN$. Then
 $$
  \lengthd{d}{B}
  =
  \sup\limits_{n\in\NNN}\lengthd{d_n}{B_n}
  =
  \lim\limits_{n\to\infty}\lengthd{d_n}{B_n}.
 $$
 Moreover, $\lengthd{d}{X}\le 1$.
\end{lemma}
\begin{proof}
First realize that for every closed (open) set $B\subseteq X$
the set $B_n=\pi_n(B)$ is a closed (open) subset of $X_n$, hence Borel measurable.
Moreover,
\begin{equation}\label{EQ:L:d-H1:1}
  \lengthd{d}{B}\ge\lengthd{d_n}{B_n}
  \qquad\text{for every }n
\end{equation}
since $\pi_n|_B:(B,d)\to (B_n,d_n)$ is Lipschitz-$1$.
We need to show that
\begin{equation}\label{EQ:L:d-H1:2}
 \lengthd{d}{B}
  \le
  \sup\limits_{n\in\NNN}\lengthd{d_n}{B_n}.
\end{equation}
We start with the case $B=X$.
To this end take any $\delta>0$ and arbitrary $n$ such that $q^n<\delta/2$.
Write $X_n$ as the union $\bigcup_{i=1}^k Y^i$ of non-overlapping subgraphs $Y^i$
such that for every $i$ the diameter $d_n(Y^i)$ is less than $\delta/2$
and the boundary $\partial(Y^i)$ does not contain any $\tilde{y}_k=f_{k,n}(\tilde{x}_k)$
($k\ge n$). Then $\lengthd{d_n}{X_n}=\sum_{i=1}^k \lengthd{d_n}{Y^i}$.
For $1\le i\le k$ put $Z^i=\pi_n^{-1}(Y^i)$; this is a subcontinuum of $X$ since
$\pi_n$ is monotone.
By (\ref{EQ:d-dn-set}), $d(Z^i) \le d_n(Y^i)+q^n<\delta$; since $X=\bigcup_{i=1}^k Z^i$
we have
$$
 \lengthd{d,\delta}{X} \le \sum_{i=1}^k d(Z^i).
$$
Since for every $k\ge n$ there is just one $i$ such that
$\tilde{y}_k\in Y^i$, Lemma~\ref{L:diam-Y-Yn} gives
$$
 \sum_{i=1}^k d(Z^i)
 \le
 \sum_{i=1}^k  \left(
   d_n(Y^i)
   + q\cdot\sum\limits_{k\ge n,\ \tilde{y}_{k}\in Y^i} \mu_k
 \right)
 \le
 \sum_{i=1}^k  d_n(Y^i)
   + q\cdot\sum\limits_{k=n}^\infty \mu_k.
$$
Thus, using (\ref{EQ:mun-q2}),
$$
 \lengthd{d,\delta}{X} \le \lengthd{d_n}{X_n} + q^n.
$$
Since this is true for every sufficiently large $n$ we have
$\lengthd{d,\delta}{X} \le \sup_n \lengthd{d_n}{X_n}$.
So (\ref{EQ:L:d-H1:2}) is proved for $B=X$
(recall that $\lengthd{d}{X}=\lim_{\delta\to 0} \lengthd{d,\delta}{X}$).

Now let $B$ be an arbitrary closed subset of $X$ and let $\eps>0$.
Take $n\in\NNN$ such that $\lengthd{d}{X} < \lengthd{d_n}{X_n}+\eps$.
Put $B_n=\pi_n(B)$ and $C=X\setminus \pi_n^{-1}(B_n) = \pi_n^{-1}(X_n\setminus B_n)$.
Then $B\cap C=\emptyset$ and
$\lengthd{d}{C}\ge \lengthd{d_n}{X_n\setminus B_n}$ by (\ref{EQ:L:d-H1:1}). So
$$
 \lengthd{d}{B}
 \le \lengthd{d}{X} - \lengthd{d}{C}
 < \left(  \lengthd{d_n}{X_n} + \eps  \right)
     -  \lengthd{d_n}{X_n\setminus B_n}
 = \lengthd{d_n}{B_n} + \eps.
$$
Thus $\lengthd{d}{X}\le \sup_n \lengthd{d_n}{B_n} + \eps$. Since $\eps>0$ was arbitrary
(\ref{EQ:L:d-H1:2}) follows.

Finally, $\lengthd{d}{X}\le 1$ follows from (\ref{EQ:L:d-H1:2}),
(\ref{EQ:metric-d1}) and (\ref{EQ:def-dn-tilde}).
\end{proof}

The following lemma describes the basic properties of
the inverse limit space $(I_\infty, d')=\invlim (I_n,\varrho_n)$ (recall the definition
of $d'$ in (\ref{EQ:d'-def})).
The fact that $I_\infty$ is an arc follows from general results. Indeed,
$I_\infty$ is a locally connected continuum since $I_n$'s are such and the bonding maps
$\varrho_n$ are monotone (see Theorem~\ref{T:monotoneInvLim1}). Moreover,
$I_\infty$ is hereditarily unicoherent since every $I_n$ is such
\cite[2.1.26]{Macias}. So $I_\infty$ is
a dendrite. Since $I_\infty$, being an arc-like continuum, does not contain a triod
\cite[2.1.41]{Macias}, it must be an arc. In our simple case, however, we can easily
prove this fact ``from scratch''.

\begin{lemma}\label{L:Iinfty-props} Put
$\alpha=\sup_n\alpha_n=\lim_n\alpha_n$ (recall that
$\alpha_n$ is the length of $I_n=[0,\alpha_n]$).
Then $\alpha<\infty$ and the following hold:
\begin{enumerate}
    \item[(a)] $d'$ is a metric on $I_\infty$ compatible with the topology;
    \item[(b)] for every $s=(s_n)_{n}, s'=(s_n')_{n}\in I_\infty$ we have
     $$
      d'(s, s')=\lim\limits_{n\to\infty} \abs{s_n-s_n'}
     $$
  \item[(c)] the projection maps $\pi'_n:(I_\infty,d')\to I_n$ ($n\in\NNN$) are
   Lipschitz-$1$;
    \item[(d)] the map
   $$
    \eta:(I_\infty,d') \to [0,\alpha],
    \quad
    s=(s_n)_{n\in\NNN} \ \mapsto\  t
    =\lim\limits_{n\to\infty} s_n  =\sup\limits_{n\in\NNN} s_n
   $$
   is an isometry;
  \item[(e)] for every subcontinuum $J$ of $I_\infty$ it holds that
   $$
    \lengthd{d'}{J}
    = \lim_{n\to\infty} \abs{J_n}
    = \sup_{n\in\NNN} \abs{J_n},
    \qquad\text{where } J_n=\pi_n'(J) \text{ for } n\in\NNN.
   $$
\end{enumerate}
\end{lemma}
\begin{proof}
 Since $\alpha_n\le \alpha_{n-1}+ 2p\cdot \lengthd{d_n}{\tilde{X}_n} < \alpha_{n-1} + q\cdot \mu_{n-1}$
 by (\ref{EQ:def-dn-tilde}),
 the finiteness of $\alpha$ follows immediately from
 (\ref{EQ:mun-q}).
 The assertion (a) can be proved similarly as the fact that $d$ is a compatible metric on
  $X$ (see Lemma~\ref{L:metric-d}). The assertion (b) follows from (\ref{EQ:d'-def}) and
  the fact that
  every $\varrho_n:I_{n+1}\to I_n$ is Lipschitz-$1$ (indeed, if $0\le s'<s\le \alpha_{n+1}$
  then $0\le \varrho_n(s) - \varrho_n(s') \le s-s'$).
  Since (c) is immediate from the definition (\ref{EQ:d'-def}) of
  $d'$ and (e) follows from (d) (indeed, $\lengthd{d'}{J}=\abs{\eta(J)}$),
  we only need to show (d). First realize that every
  $s=(s_n)_{n\in\NNN}\in I_\infty$ is a non-decreasing sequence
  (in fact, every $\varrho_n$ is Lipschitz-$1$ and $\varrho_n(0)=0$) bounded
  from above by $\alpha$, hence $\eta(s)$ is well defined.
  Continuity of $\eta$ is trivial. Since $\eta(0,0,\dots)=0$ and
  $\eta(\alpha_1,\alpha_2,\dots)=\alpha$, $\eta$ is surjective. It remains to show that
  $\eta$ is an isometry. But this is trivial by (b):
  for any $s=(s_n)_{n}, s'=(s_n')_{n}\in I_\infty$ we have
  $d'(s,s')=\lim_n \abs{s_n-s_n'}=\abs{\lim_n s_n - \lim_n s_n'}$. The proof is finished.
\end{proof}

\subsection{Properties of the maps $g_n$, $h_n$}\label{SS:proof-props-gn-hn}
In what follows we prove that the lengths of $g_n$-images and $h_n\circ g_n$ images
of any closed interval $J\subseteq I_n$ are bounded from below by some constant multiple
of the length of $J$, where the constant
does not depend on $n,J$.
Till the end of this subsection fix $n\in\NNN$.

\begin{lemma}\label{L:progs-gn-hn-1vert}
 Let $J\subseteq I_n$ be a compact interval and let $Y=g_n(J)$, $L=h_n(Y)$.
 If $Y$ contains at most one vertex of $X_n$ then
 $$
  \lengthd{d_n}{Y} \ge \frac 12\cdot\abs{J}
  \qquad\text{and}\qquad
  \abs{L}   \ge \frac 14\cdot \lengthd{d_n}{Y}.
 $$
\end{lemma}
\begin{proof}
 By the assumption we can write $J=J^{(0)}\cup J^{(1)}$ and $Y=Y^{(0)}\cup Y^{(1)}$
 such that $J^{(0)},J^{(1)}$ are non-overlapping compact intervals and, for $i=0,1$,  $Y^{(i)}=g_n(J^{(i)})$ is a free arc and $g_n|_{J^{(i)}}:J^{(i)}\to Y^{(i)}$ is an isometry
(see (\ref{EQ:gn-is-parametrization})).
For $i=0,1$ put $L^{(i)}=h_n(Y^{(i)})$. Then the inequalities
\begin{eqnarray*}
 \max \abs{J^{(i)}} \ \le& \abs{J} &\le 2\cdot \max \abs{J^{(i)}}
\\
 \max \lengthd{d_n}{Y^{(i)}} \ \le& \lengthd{d_n}{Y} &\le 2\cdot \max \lengthd{d_n}{Y^{(i)}}
\\
 \max \abs{L^{(i)}} \ \le& \abs{L}  &\
\end{eqnarray*}
together with Lemma~\ref{L:freeArc} and $\lengthd{d_n}{Y^{(i)}}=\abs{J^{(i)}}$ ($i=0,1$)
give the assertions of the lemma.
\end{proof}

\begin{lemma}\label{L:progs-gn-hn-2vert}
 Let $J\subseteq I_n$ be a compact interval and let $Y=g_n(J)$, $L=h_n(Y)$.
 Assume that $Y$ contains at least two vertices of $X_n$ and that
 $Y'=f_{n-1}(Y)$ contains at most one vertex of $X_{n-1}$. Then
 $$
  \lengthd{d_n}{Y} \ge \frac 12\cdot\abs{J}
  \qquad\text{and}\qquad
  \abs{L}   \ge \frac {1-q}{12}\cdot \lengthd{d_n}{Y}.
 $$
\end{lemma}
\begin{proof}
Put $Y'=f_{n-1}(Y)$.
If $Y\cap \tilde{X}_n$ is empty or degenerate we can proceed as in the proof of the
previous lemma, since
under this assumption $f_{n-1}|Y:Y\to Y'$ is an isometry. So assume that $Y\cap \tilde{X}_n$ is non-degenerate.
Put $J'=\varrho_{n-1}(J)$. Since $Y'=f_{n-1}(Y)$ contains at most one vertex of $X_{n-1}$,
it is either degenerate or can be written as the union of two free arcs.
Using the fact that just one point (namely the point $\tilde{x}_{n-1}$)
has non-degenerate $f_{n-1}$-preimage, we have that
$J$ and $Y$ can be written as the non-overlapping unions
$$
 J = J_0\cup \tilde{J} \cup J_1
 \qquad\text{and}\qquad
 Y = Y_0\cup \tilde{Y} \cup Y_1,
$$
where (for $i=0,1$)
\begin{itemize}
 \item $J_0\le\tilde{J}\le J_1$ are compact intervals ($J_0,J_1$ can be degenerate);
 \item $Y_i$ is either a degenerate subset of $X_n$ or a free arc in
   $X_n\setminus \interior{\tilde{X}_n}$;
 \item $g_n|_{J_i}:J_i\to Y_i$ is an isometry;
 \item $\tilde{Y}=g_n(\tilde{J})\subseteq \tilde{X}_n$ is non-degenerate;
 \item $g_n|_{\tilde{J}}:\tilde{J}\to\tilde{Y}$ is admissible.
\end{itemize}
(Notice that it can happen that $Y_i$'s contain more than one vertex of $X_n$; however, every vertex
contained in the interior of $Y_i$ has order $2$, so $g_n$ ``goes-through'' it, see the definition of an admissible path.)

For $i=0,1$ put $L_i=h_n(Y_i)$; we have
$$
 \lengthd{d}{Y_i}=\abs{J_i}
 \qquad\text{and}\qquad
 \abs{L_i}\ge\frac 12\cdot\lengthd{d}{Y_i}.
$$
For $\tilde{L}=h_n(\tilde{Y})$ Lemma~\ref{L:graph:admissible2} gives
$$
 \lengthd{d_n}{\tilde{Y}}\ge \frac 12\abs{\tilde{J}}
 \qquad\text{and}\qquad
 \abs{\tilde{L}} \ge \frac{1-q}{6} \lengthd{d_n}{\tilde{Y}}.
$$
Now the first inequality of Lemma~\ref{L:progs-gn-hn-2vert} follows since
$$
   \lengthd{d_n}{Y}
 \ge
   \lengthd{d_n}{\tilde{Y}}
   + \max_i \lengthd{d_n}{Y_i}
 \ge
   \lengthd{d_n}{\tilde{Y}}
   + \frac 12 \cdot\left(   \lengthd{d_n}{Y_1} + \lengthd{d_n}{Y_2} \right).
$$

To show the second inequality it suffices to use (\ref{EQ:max-sum}) and the fact that either
one of $J_0, J_1$ is degenerate or $\tilde{Y}=\tilde{X}_n$ and, in the latter case,
$\abs{\tilde{L}} \ge \frac{1-q}{2}\cdot \lengthd{d_n}{\tilde{Y}}$ by Lemma~\ref{L:graph:admissible2}.
\end{proof}

\begin{lemma}\label{L:progs-gn-hn-nm}
Let $J_n\subseteq I_n$ be a compact interval and let $Y_n=g_n(J_n)$, $L_n=h_n(Y_n)$.
For $1\le m< n$ put
$$
 J_m=\varrho_{n,m}(J_n),
 \qquad
 Y_m=f_{n,m}(Y_n)=g_m(J_m)
 \qquad\text{and}\qquad
 J_m=h_m(Y_m).
$$
Let $m$ be such that $Y_m$ contains at least two vertices of $X_m$.
Then
$$
 \abs{J_m} \le \abs{J_n} < \frac{1}{1-q}\cdot \abs{J_m},
 \qquad
 \lengthd{d_m}{Y_m} \le \lengthd{d_n}{Y_n} < \frac{1}{1-q}\cdot \lengthd{d_m}{Y_m}
$$
and
$$
 \frac{1-3q}{1-q}\cdot\abs{L_m} < \abs{L_n}.
$$
\end{lemma}
\begin{proof}
If $m=n$ there is nothing to prove; hence we may assume that $m\le n-1$.
The first and the third inequalities are immediate consequences of the fact
that $\varrho_{n,m}$ and $f_{n,m}$ are Lipschitz-$1$.
By  the definition of $\varrho_n$ we have that
$\abs{J_n} = \abs{J_{n-1}} + \sum_{i=1}^p \abs{J\cap K_i}$,
where the intervals $K_1,\dots,K_p$ are such that every restriction
$g_n|_{K_i}:K_i\to \tilde{X}_n$ is a fully-admissible map,
see (\ref{EQ:def-In}) and
(\ref{EQ:def-gn}). The simple estimate
$
 \abs{J_n\cap K_i}
 \le \abs{K_i}
 \le 2\cdot \lengthd{d_n}{\tilde{X}_n}
$
together with (\ref{EQ:def-dn-tilde}) gives that
$\abs{J_n} < \abs{J_{n-1}} + q\cdot \mu_{n-1}$.
By (\ref{EQ:fn-isometry}) we have that
$\lengthd{d_{n-1}}{Y_{n-1}}=\lengthd{d_n}{Y_n\setminus\interior{\tilde{X}_n}}$; so, again by (\ref{EQ:def-dn-tilde}),
$\lengthd{d_n}{Y_n}   <  \lengthd{d_{n-1}}{Y_{n-1}} + q\cdot \mu_{n-1}$.
Repeating the previous arguments with $n$ replaced by $n-1,n-2,\dots, m+1$ gives
that
$\abs{J_n} < \abs{J_{m}} + q\cdot (\mu_{n-1}+\mu_{n-2} + \dots + \mu_m)$
and
$\lengthd{d_n}{Y_n}  <   \lengthd{d_{m}}{Y_{m}} + q\cdot (\mu_{n-1}+\mu_{n-2} + \dots + \mu_m)$.
Now the second and the fourth inequality immediately follows from (\ref{EQ:mun-q})
and the fact that $\mu_m\le \lengthd{d_m}{Y_m}\le \abs{J_m}$
since $Y_m$ contains an edge of $X_m$.

It remains to show the fifth inequality. Let $x',y'\in Y_{n-1}$ be such
that $\abs{L_{n-1}}=d_{n-1}(a_{n-1},x')-d_{n-1}(a_{n-1},y')$.
Take any $x\in Y_n\cap f_{n-1}^{-1}(x')$, $y\in Y_n\cap f_{n-1}^{-1}(y')$. Then, by
(\ref{EQ:dn}),
$$
 \abs{L_n}
 \ge d_{n}(a_{n},x)-d_{n}(a_{n},y)
 > d_{n-1}(a_{n-1},x')  -  \left(d_{n-1}(a_{n-1},y') + q\cdot \mu_{n-1} \right)
$$
so $\abs{L_n}> \abs{L_{n-1}} - q\cdot \mu_{n-1}$.
Continue in this fashion to obtain
$$
 \abs{L_n}
 > \abs{L_m} - q\cdot (\mu_{n-1}+\mu_{n-2} + \dots + \mu_m)
 > \abs{L_m} - \frac{q}{1-q}\cdot \mu_m.
$$
Since $Y_m$ contains an edge $E$ of $X_m$, Lemma~\ref{L:freeArc} gives
that
$\abs{L_m}
 \ge \abs{h_{m}(E)}
 \ge (1/2) \cdot\lengthd{d_m}{E}$, i.e.~$\abs{L_m}\ge (1/2)\cdot \mu_m$.
So
$$
 \abs{L_n}
 > \abs{L_{m}} - \frac{2q}{1-q}\cdot \abs{L_m}
 = \frac{1-3q}{1-q}\cdot\abs{L_m}.
$$
\end{proof}

Combining Lemmas~\ref{L:progs-gn-hn-1vert}--\ref{L:progs-gn-hn-nm} gives the following estimates.

\begin{lemma}\label{L:gn-props-base}
The map $g_n:I_n\to (X_n,d_n)$ is a Lipschitz-$1$ surjection.
Moreover, for every compact interval $J\subseteq I_n$
and for $Y=g_n(J)$, $L=h_n(Y)$ it holds that
$$
  \lengthd{d_n}{Y} \ge \frac{1-q}{2}\cdot\abs{J}
  \qquad\text{and}\qquad
  \abs{L}   \ge \frac {1-4q}{12}\cdot \lengthd{d_n}{Y}.
$$
\end{lemma}
\begin{proof}
The fact that $g_n$ is a Lipschitz-$1$ surjection is
an immediate consequence of (\ref{EQ:gn-is-parametrization}).
To prove the second part of the assertion take any non-degenerate compact interval
$J\subseteq I_n$ and put $Y=g_n(J)$, $L=h_n(Y)$.
If $Y$ contains at most one vertex of $X_n$ we can use Lemma~\ref{L:progs-gn-hn-1vert}.
So assume that $Y$ contains at least two vertices of $X_n$.
As in Lemma~\ref{L:progs-gn-hn-nm}, for every $1\le m\le n$ put
$J_m=\varrho_{n,m}(J)$,
$Y_m=f_{n,m}(Y_n)=g_m(J)$ and $J_m=h_m(Y)$.
Let $m\ge 1$ be the smallest integer such that $Y_m$ contains at least two vertices of $X_m$.
Using Lemma~\ref{L:graph:admissible2} (if $m=1$)
or Lemma~\ref{L:progs-gn-hn-2vert} (if $m\ge 2$) we obtain that
$\lengthd{d_m}{Y_m}\ge \frac{1}{2}\cdot \abs{J_m}$
and $\abs{L_m}\ge \frac{1-q}{12}\cdot \lengthd{d_m}{Y_m}$.
Now Lemma~\ref{L:progs-gn-hn-nm} gives
$$
 \frac{\lengthd{d_n}{Y}}{\abs{J}}
 \ge
 (1-q) \cdot \frac{\lengthd{d_m}{Y_m}}{\abs{J_m}}
 \ge
 \frac{1-q}{2}
$$
and
$$
 \frac{\abs{L}}{\lengthd{d_n}{Y}}
 \ge
 (1-3q) \cdot \frac{\abs{L_m}}{\lengthd{d_m}{Y_m}}
 \ge
 (1-3q)\cdot\frac{1-q}{12}.
$$
The desired inequalities follow.
\end{proof}

\subsection{Properties of the maps $g$, $h$}\label{SS:proof-props-g-h}
Recall that $g:I_\infty\to X$ is given by
$g(s_1,s_2,\dots)=(g_1(s_1),g_2(s_2),\dots)$.

\begin{lemma}\label{L:g-props-base}
The map $g:(I_\infty,d')\to (X,d)$ is a Lipschitz-$1$ surjection.
\end{lemma}
\begin{proof}
By the definitions of $d,d'$ and the fact that
the maps $g_n$ are Lipschitz-$1$ we have that for every $s=(s_n)_n$, $t=(t_n)_n\in I_\infty$
$$
 d(g(s),g(t))
 = \sup_n d_n(g_n(s_n),g_n(t_n))
 \le \sup_n \abs{s_n - t_n}
 = d'(s,t).
$$
So $g$ is Lipschitz-$1$. The surjectivity of $g$ follows from the surjectivity
of the maps $g_n$ by Theorem~\ref{T:inducedMap}.
\end{proof}

\begin{lemma}\label{L:g-props-length}
Let $J$ be a subcontinuum of $I_\infty$ and let $Y=g(J)$, $L=h(Y)$. Then
$$
  \lengthd{d}{Y} \ge \frac{1-q}{2}\cdot\lengthd{d'}{J}
  \qquad\text{and}\qquad
  \abs{L}   \ge \frac {1-4q}{12}\cdot \lengthd{d}{Y}.
$$
\end{lemma}
\begin{proof}
For every $n$ put $J_n=\pi_n'(J)$, $Y_n=\pi_n(Y)$ and $L_n=h_n(Y_n)$.
By Lemmas~\ref{L:d-H1} and \ref{L:Iinfty-props} we have
$$
 \lengthd{d'}{J} = \lim_{n\to\infty} \abs{J_n}
 \qquad\text{and}\qquad
 \lengthd{d}{Y} = \lim_{n\to\infty} \lengthd{d_n}{Y_n}.
$$
So the first inequality immediately follows from Lemma~\ref{L:gn-props-base}.
If we prove that
\begin{equation}\label{EQ:L:g-props-length:1}
 \abs{L} = \lim_{n\to\infty} \abs{L_n}
\end{equation}
then we analogously obtain the second inequality. (In fact, just an inequality
in (\ref{EQ:L:g-props-length:1}) is sufficient.)
But (\ref{EQ:L:g-props-length:1}) is an immediate consequence of the definitions of $h$ and $d$.
Indeed,
$$
 \abs{L} = \sup_{y,y'\in Y} d(a,y)-d(a,y'),
 \qquad
 \abs{L_n} = \sup_{y,y'\in Y} d_n(a_n,\pi_n(y))-d_n(a_n,\pi_n(y')),
$$
so (\ref{EQ:d-dn}) immediately gives (\ref{EQ:L:g-props-length:1}).
\end{proof}

\subsection{Summarization}\label{SS:proof-main-thm}
Here we prove Lemma~\ref{P:props-g-h}, which is the main result of
Section~\ref{S:proofOfMainThms}.
We start with some simple observations concerning cut points.

\begin{lemma}\label{L:uncountable-cut} Let $X=\invlim (X_n,f_n)$ be the monotone
inverse limit of continua and let $\pi_n:X\to X_n$ ($n\in\NNN$) be the natural projections.
Take any two points $x,y\in X$ and put $x_n=\pi_n(x)$, $y_n=\pi_n(y)$ ($n\in\NNN$).
Then
\begin{enumerate}
    \item[(a)] $\pi_n \Cut_X(x,y)\subseteq  \Cut_{X_n}(x_n,y_n)$ for every $n\in\NNN$;
    \item[(b)] if $\Cut_X(x,y)$ is uncountable then
     $\Cut_{X_n}(x_n,y_n)$ is uncountable for every sufficiently large $n$;
    \item[(c)] if $\Cut_X(x,y)$ is countable and $X$ is rational
         then every $\Cut_{X_n}(x_n,y_n)$ is countable.
\end{enumerate}
\end{lemma}
\begin{proof}
Put $C=\Cut_X(x,y)$ and $C_n=\Cut_{X_n}(x_n,y_n)$ for $n\in \NNN$.

(a) If $z\in X$ is such that $z_n=\pi_n z\not\in C_n$ then
there is a connected subset $D_n$ of $X_n\setminus \{z_n\}$ containing both
$x_n,y_n$. Then $D=\pi_n^{-1}(D_n)$ is a connected subset of $X\setminus\{z\}$
containing both $x,y$, hence $z\not\in C$.

(b) Assume that $C$ is uncountable.
For every $z\ne z'$ from $X$ there is $n_{z,z'}\in \NNN$
such that $\pi_n z \ne \pi_n z'$ for every $n\ge n_{z,z'}$.
Since $C$ is uncountable there is $n_0$ such that
$n_{z,z'}=n_0$ for uncountably many pairs of distinct points $z,z'$ from $C$.
Hence, by (a), $C_n\supseteq \pi_n C$ is uncountable for every $n\ge n_0$.

(c) If $C_n$ is uncountable for some $n$ then $D=\pi_n^{-1}(C_n)$ has uncountably
many components and every component of $D$ separates $x,y$. Since $X$
is rational, only countably many components of
$D$ are non-degenerate
(see \cite[Theorem~51.IV.5]{Kur2})
and so uncountably many components of $D$ are singletons. Thus
uncountably many points of $D$ separate $x,y$, i.e.~$C$ is
uncountable.
\end{proof}

\begin{lemma}\label{P:props-g-h}
There are constants $0<\gamma<\Gamma$ such that for any $\delta>0$ the following hold:
For any non-degenerate totally regular continuum $X$ and any two points $a,b$ of $X$
there are a compatible convex metric $d$ on $X$ and maps
$g:[0,\alpha]\to X$, $h:X\to [0,\beta]$ with the following properties:
\begin{enumerate}
  \item[(a)] $g(0)=a$, $g(\alpha)=b$ and $h(a)=0$;
  \item[(b)] $g,h$ are Lipschitz-$1$ surjections;
    \item[(c)] $\gamma\cdot \abs{J} \le \lengthd{d}{g(J)}\le \Gamma\cdot\abs{h\circ g(J)}$
     for every closed subinterval $J$ of $[0,\alpha]$;
  \item[(d)] $\lengthd{d}{X} \in[1-\delta,1]$,
    $\lengthd{d}{X}\le \alpha \le 2\cdot \lengthd{d}{X}$
    and $(1/2 -\delta)\cdot\lengthd{d}{X}\le \beta \le \lengthd{d}{X}$.
 \end{enumerate}
 Moreover, if $\Cut(a,b)$ is uncountable then a metric $d$ and maps $g,h$
 can be chosen such that also:
 \begin{enumerate}
    \item[(e)] $h(b)=\beta$;
    \item[(f)] $d(a,b)>(1-\delta)\cdot \lengthd{d}{X}$.
 \end{enumerate}
\end{lemma}

\begin{proof}
Let $0<\gamma<\frac 12$, $\Gamma>24$ and $0<\delta\le \frac 12$. Take $0<q<1$ such that
$$
 \frac{1-q}{2} \ge\gamma,
 \qquad
 \frac{1-4q}{12} \ge \frac{2}{\Gamma}
 \qquad\text{and}\qquad
 2q<\delta.
$$
For a non-degenerate totally regular continuum $X$ and $a,b\in X$
construct a compatible convex metric $d$ on $X$ and maps $g:[0,\alpha]\to X$,
$h:X\to [0,\beta]$ as in Section~\ref{SS:construction};
particularly, $h(x)=d(a,x)$ for $x\in X$ and $\beta=\max_{x\in X} d(a,x)$.
By Lemma~\ref{L:g-props-base}, $g$ is a Lipschitz-$1$ surjection; since $h$ is such trivially,
we have (b). The property (a) is also immediate: $h(a)=d(a,a)=0$ and
$g_n(0)=a_n$, $g_n(\alpha_n)=b_n$ for every $n$ by (\ref{EQ:gn-is-parametrization}),
so $g(0)=a$ and $g(\alpha)=b$. The property (c) follows from Lemma~\ref{L:g-props-length} and
the choice of $q$.

To finish the proof of the first part we have to show (d).
The inequalities $1-\delta\le \lengthd{d}{X} \le 1$ follow from (\ref{EQ:metric-d1})
and Lemma~\ref{L:d-H1}.
Since $g,h$ are Lipschitz-$1$ surjections
we immediately have $\beta\le \lengthd{d}{X}\le\alpha$.
Lemma~\ref{L:graph:admissible2} applied to $J=[0,\alpha_n]$, $\kappa=g_n$ and $a=a_n$
gives that
$$
  \alpha_n\le 2\cdot \lengthd{d_n}{X_n}
  \qquad\text{and}\qquad
  \abs{h_n(X_n)}
  \ge \frac{1-q}{2}\cdot \lengthd{d_n}{X_n}.
$$
Thus, by Lemmas~\ref{L:Iinfty-props} and \ref{L:d-H1},
$$
 \alpha
 =\lim_{n\to\infty} \alpha_n
 \le 2\cdot \lim_{n\to\infty} \lengthd{d_n}{X_n}
 =2\cdot \lengthd{d}{X}
$$
and
$$
 \beta
 =\lim_{n\to\infty} \abs{h_n(X_n)}
 \ge
 \frac{1-q}{2}\cdot \lim_{n\to\infty}\lengthd{d_n}{X_n}
 = \frac{1-q}{2}\cdot \lengthd{d}{X}
 > \left(\frac 12 -\delta\right) \cdot \lengthd{d}{X}.
$$

Now assume that $a,b$ are such that the set $\Cut(a,b)$ of points
which separate them is uncountable (hence $a\ne b$). By Lemma~\ref{L:uncountable-cut}
for every sufficiently large $n$ the set $\Cut_{X_n}(a_n,b_n)$ of points in $X_n$ separating $a_n,b_n$ is uncountable; without loss of generality we may assume that $\Cut_{X_1}(a_1,b_1)$
is uncountable. Hence there is a free arc $A$ in $X_1$ such that
$X_1\setminus\interior{A}$ has exactly two components, one of which contains $a_1$
and the other one contains $b_1$. We may assume (adding two vertices of order $2$ if necessary) that $A$ is an edge of $X_1$. We modify the construction
of the metric $d_1$ on $X_1$ such that the longest edge is $E_0=A$
(see the proof of Lemma~\ref{L:graph:admissible2}).
Then
$$
 d_1(a_1,b_1)
 \ge \lengthd{d_1}{A}
 \ge (1-q)\cdot \lengthd{d_1}{X_1}.
$$
Lemma~\ref{L:d-H1} and (\ref{EQ:fn-isometry}), (\ref{EQ:def-dn-tilde}), (\ref{EQ:mu-n}) give
\begin{equation*}
\begin{split}
 \lengthd{d}{X}
 &= \lim_{n\to\infty} \lengthd{d_n}{X_n}
 = \lim_{n\to\infty} \left[
     \lengthd{d_1}{{X}_1} +
     \lengthd{d_2}{\tilde{X}_2} + \dots +
     \lengthd{d_n}{\tilde{X}_n}
   \right]
\\
 &\le \frac{1}{1-q}\cdot \lengthd{d_1}{{X}_1}.
\end{split}
\end{equation*}
So, by (\ref{EQ:d-dn}),
$$
 d(a,b)
 \ge d_1(a_1,b_1)
 \ge (1-q)^2 \cdot \lengthd{d}{X}
 > (1-2q) \cdot \lengthd{d}{X}
$$
and (f) is satisfied.

To obtain also (e) we must replace $h$ by
$$
 \tilde{h}=\lambda\circ h:X\to[0,\tilde{\beta}]
 \qquad
 \text{where } \tilde{\beta}=d(a,b) \text{ and }
 \lambda(s) = \begin{cases}
  s  &\text{if } s\le \tilde{\beta};
  \\
  2\tilde{\beta}-s   &\text{if } \tilde{\beta}< s\le {\beta}.
 \end{cases}
$$
Notice that $\tilde{\beta}=d(a,b)>(1-\delta)\cdot\lengthd{d}{X}
\ge (1-\delta)\cdot d(X)\ge \frac 12 \beta$, so $\lambda(s)\in[0,\tilde\beta]$
for every $s\in[0,\beta]$. Notice also that
$\tilde\beta > (\frac 12-\delta)\cdot\lengthd{d}{X}$; thus to
prove that the triple $d,g,\tilde{h}$ satisfies (a)--(f) we only need to show that $\lengthd{d}{g(J)}\le \Gamma\cdot
\abs{\tilde{h}\circ g(J)}$ for every closed subinterval $J$ of $[0,\alpha]$, since
the other properties are satisfied trivially.
To this end fix a closed subinterval $J\subseteq [0,\alpha]$ and put $Y=g(J)$.
Using
$$
 \abs{\tilde{h}(Y)}
 \ge
 \max\{  \abs{h(Y)\cap[0,\tilde{\beta}]},\  \abs{h(Y)\cap[\tilde{\beta},\beta]}\}
$$
and Lemma~\ref{L:g-props-length} we immediately have
$$
 \Gamma\cdot
 \abs{\tilde{h}(Y)}
 \ge
 \frac{\Gamma}{2} \cdot \abs{h(Y)}
 \ge
 \frac {12}{1-4q} \cdot \abs{h(Y)}
 \ge \lengthd{d}{Y}.
$$
Hence the proof is finished.
\end{proof}

\section{Length-expanding Lipschitz maps
from/to the interval}
\label{S:varphi}

The following proposition provides the key tool
for constructing LEL maps.
Basically it is just a reformulation of Lemma~\ref{P:props-g-h}.

\begin{proposition}\label{T:main}
 There are constants $0<\gamma<\Gamma$ and $L>1$ such that the following hold:
 For every non-degenerate totally regular continuum $X$ and every two points $a,b\in X$
 there are a compatible convex metric $d$ on $X$ and maps $\varphi:\III\to X$, $\psi:X\to \III$
 with the following properties:
 \begin{enumerate}
    \item[(a)] $\varphi(0)=a$, $\varphi(1)=b$ and $\psi(a)=0$;
    \item[(b)] $\varphi,\psi$ are Lipschitz-$L$ surjections;
    \item[(c)] $\gamma\cdot \abs{J} \le \lengthd{d}{\varphi(J)}\le \Gamma\cdot\abs{\psi\circ \varphi(J)}$
     for every closed subinterval $J$ of $\III$;
  \item[(d)] $\lengthd{d}{X}= 1$.
 \end{enumerate}
 Moreover, if $\Cut(a,b)$ is uncountable then for any $\delta>0$ a metric $d$ and maps $\varphi,\psi$
 can be chosen such that it also holds:
 \begin{enumerate}
    \item[(e)] $\psi(b)=1$;
    \item[(f)] $d(a,b)>1-\delta$.
 \end{enumerate}
\end{proposition}

\begin{proof}
Take any $L>2$ and let $0<\gamma<\Gamma$ be constants from
Lemma~\ref{P:props-g-h}. Fix a non-degenerate totally regular
continuum $X$, a pair $a,b\in X$ and a
positive real $\delta$; we may assume that $2/(1-2\delta)<L$.
We give the proof only in the case when $\Cut(a,b)$ is uncountable; the other case can be described analogously.

Let $\tilde{d}$ be a convex metric on $X$ and $g:[0,\alpha]\to X$, $h:X\to
[0,\beta]$ be maps satisfying (a)--(f) from
Lemma~\ref{P:props-g-h}. Now define
$d:X\times X\to\RRR$, $\varphi:\III\to X$ and $\psi:X\to \III$ by
$$
 d(x,y) = \frac 1c\cdot \tilde{d}(x,y), \qquad
 \varphi(t)=g(\alpha t)
 \qquad\text{and}\qquad
 \psi(x)=\frac{1}{\beta}\cdot  h(x),
$$
where $c=\lengthd{\tilde{d}}{X}$.
Then (a) and (d)--(f) are immediately satisfied. Since
$$
  \Lip_d(\varphi)=\frac{\alpha}{c}\cdot\Lip_{\tilde d}(g)\le 2<L
  \qquad\text{and}\qquad
  \Lip_d(\psi)=\frac{c}{\beta}\cdot\Lip_{\tilde d}(h)\le\frac{2}{1-2\delta}<L,
$$
also (b) is fulfilled.
The property (c) follows from
$$
 \lengthd{d}{\varphi(J)}
 \ge \frac{\gamma \alpha}{c}\cdot \abs{J}
 \qquad\text{and}\qquad
 \Gamma\cdot  \abs{\psi\circ\varphi(J)}
 \ge \frac{c}{\beta}\cdot \lengthd{d}{\varphi(J)}
$$
and from $\alpha\ge c\ge \beta$.
\end{proof}

\begin{corollary}\label{C:tentLikeInterval}
Every non-degenerate totally regular continuum $X$, endowed with a suitable
convex metric $d$ and a dense systems $\CCc$ of subcontinua of $X$,
admits LEL-maps $\tilde{\varphi}:(I,d_I,\CCc_I)\to (X,d,\CCc)$ and
$\tilde{\psi}:(X,d,\CCc)\to (I,d_I,\CCc_I)$.
\end{corollary}
\begin{proof}
Fix arbitrary $a,b\in X$; let $d,\varphi,\psi$ be as in Proposition~\ref{T:main}.
Put $\CCc=\varphi(\CCc_I)$; this is a dense system by Lemma~\ref{L:CX}.
Let $f_k$ be the map from Lemma~\ref{L:tentLike_fk}, where $k\ge 3$ is
such that $\varrho=\gamma k/2> 1$. Then the map
$\tilde{\varphi}=\varphi\circ f_k:I\to X$ is $(\varrho, kL)$-LEL.
Analogously, if $k'\ge 3$ is such that $\varrho'=k'/(2\Gamma)>1$
then $\tilde{\psi}=f_{k'}\circ \psi:X\to I$ is $(\varrho', k'L)$-LEL.
\end{proof}

Notice that from the proofs of Lemma~\ref{P:props-g-h} and
Proposition~\ref{T:main} we can see that to fulfill only the conditions (a)--(d) we can find
$d,\varphi,\psi$ such that
$$
 \psi(x)=c\cdot d(a,x)
 \quad\text{for every}\quad
 x\in X,
$$
where $c$ is a constant. One can also see that any constants
$0<\gamma<\frac 12$, $\Gamma>24$ and $L>2$ are suitable in
Proposition~\ref{T:main}. Derivation of the ``best'' values for $\gamma,\Gamma$ and
$L$ is out of the scope of this paper. However, we can at least say
that $L$ and the ratio $\Gamma/\gamma$ cannot be arbitrarily close
to $1$. In fact, if $X$ is the $3$-star then easy arguments show
that we must have $\Gamma/\gamma \ge 3$. Further, if $X=(X,d)$ is a
simple closed curve of length $1$ then, for any $\psi:X\to \III$
from Proposition~\ref{T:main}, we
can write $X$ as the union $A\cup B$ of two non-overlapping arcs such
that $\psi(A)=\psi(B)=\III$; so $L\ge\Lip(\psi)\ge 2$.

The following example shows that in the second part of Proposition~\ref{T:main} one
cannot replace the assumption $\Cut(a,b)$ is uncountable by
$\Cut(a,b)$ is infinite.

\begin{example}
Take an integer $p\ge 3$,
put $a=(-1,0)$, $b=(1,0)$, $a_0=(0,0)$, $a_k=(1-2^{-k},0)$, $a_{-k}=-a_k$ ($k\in\NNN$) and define
a continuum $X_p\subseteq\RRR^2$ by
$$
 X_p = \bigcup_{k\in\ZZZ} G_k \cup \{a,b\}
$$
where every $G_k$ ($k\in\ZZZ$) is a graph with exactly two vertices
$a_{k-1},a_k$, these vertices have order $p$ (in $G_k$) and $G_k\cap G_l$ is empty for
$l>k+1$ and is equal to $\{a_k\}$ for $l=k+1$; see
Figure~\ref{Fig:infChain} for $p=3$. In this case $a,b$ are end points of $X_p$  (so $\Cut(a,b)$
is infinite), but neither (e) nor (f) can be fulfilled for small
$\delta$.

 \begin{figure}[ht!]
   \includegraphics[width=0.8\textwidth]{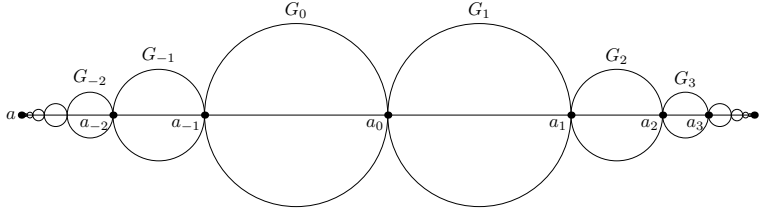}
   \caption{The continuum $X_3$}
   \label{Fig:infChain}
 \end{figure}

To show this realize that
$\lengthd{d}{X_p} \ge p\cdot d(a,b)$ for any convex metric $d$ on $X_p$;
indeed, $X_p$ is the union of $p$ arcs with ends $a,b$ (so the length of any of them
is greater than or equal to $d(a,b)$) and with countable intersections.
So immediately (f) is not true for $\delta<1-\frac 1p$.
Moreover, since $\psi$ is Lipschitz-$L$ and $\psi(a)=0$
 we have that
$$
 \psi(b)=\psi(b)-\psi(a)
 \le L\cdot d(a,b)
 \le L\cdot \frac{\lengthd{d}{X_p}}{p}
 = \frac{L}{p}
$$
which is smaller than $1$ for $p>L$.
So also (e) is not true. Notice that any metric $d$ satisfying (a)--(d) must be such
that the diameter of $X_p$ is approximately $p$-times larger than the distance of $a,b$;
so for some $k$ the shortest edge of $G_k$ must be ``very small'' when compared to the longest
one.
\end{example}

\begin{remark}\label{R:replacementOf(d)(f)}
If we replace the metric $d$
from Proposition~\ref{T:main} by $d'=c\cdot d$ (where $c>0$ is a constant),
the Lipschitz constants of $\varphi,\psi$ change to
$\Lip_{d'}(\varphi)=c\cdot \Lip_{d}(\varphi)$,
$\Lip_{d'}(\psi)=(1/c) \cdot \Lip_{d}(\psi)$.
So instead of the conditions (d) and (f) we can have
 \begin{enumerate}
    \item[(d')] $\lengthd{d}{X}<1/(1-\delta)$;
    \item[(f')] $d(a,b)=1$.
 \end{enumerate}
\end{remark}

\section{Proofs of the main results}\label{S:applications}
Now we are ready to prove the main results of the paper stated in the introduction.
For convenience we repeat the
statements of them.

\renewcommand{\themainresult}{\ref{T:MAINa}}
\begin{theoremA}
\thmMainA{}
\end{theoremA}
\begin{proof}
Let $\gamma,\Gamma$ and $L$ be constants from
Proposition~\ref{T:main}. Let $X$ be a non-degenerate totally regular continuum
and $a,b$ be two points of $X$.
Put $\delta=1/2$ and fix a metric $d=d_{X,a,b}$ on $X$ and
maps $\varphi_{X,a,b}:\III\to X$, $\psi_{X,a,b}:X\to \III$ satisfying (a)--(d)
(or (a)--(f) if $\Cut_X(a,b)$ is uncountable) from Proposition~\ref{T:main}.
Recall that
$\lengthd{d}{X}=1$ and, provided $\Cut_X(a,b)$ is uncountable, $d(a,b)>1/2$.
By Lemma~\ref{L:CX}, $\CCc_{X,a,b}=\varphi_{X,a,b}(\CCc_I)$ 
is a dense system of subcontinua of $X$.

Let
$k,l\ge 3$ be the smallest odd integers such that $\gamma k/2\ge \varrho$
and $l/(2\Gamma)\ge\varrho$.
Put $L_\varrho=2L\varrho\cdot(1+\max\{\Gamma,1/\gamma\})>1$.
As in the proof of Corollary~\ref{C:tentLikeInterval},
the maps
$\varphi=\varphi_{X,a,b}\circ f_{k}:I\to X$ and $\psi=f_{l}\circ \psi_{X,a,b}:X\to I$
are $(\varrho,L_\varrho)$-LEL.
Since $k,l$ are odd we have $\varphi(0)=\varphi_{X,a,b}(0)=a$,
 $\varphi(1)=\varphi_{X,a,b}(1)=b$, $\psi(a)=\psi_{X,a,b}(a)=0$ and,
 provided  $\Cut_X(a,b)$ is uncountable, $\psi(b)=\psi_{X,a,b}(b)=1$.
\end{proof}

\renewcommand{\themainresult}{\ref{T:MAIN}}
\begin{theoremA}
\thmMain{}
\end{theoremA}

\begin{proof}
The theorem follows from Theorem~\ref{T:MAINa} and
Lemma~\ref{L:tentLikeComposition}.
\end{proof}

\renewcommand{\themainresult}{\ref{C:mainCont}}
\begin{corollaryA}
\corCont{}
\end{corollaryA}

\begin{proof}
This immediately follows from Theorem~\ref{T:MAIN} and
Proposition~\ref{P:tentLikeIsExact}.
\end{proof}

\renewcommand{\themainresult}{\ref{C:mainUnion}}
\begin{corollaryA}
\corUnion{}
\end{corollaryA}
\begin{proof}
 Let $X=\bigsqcup_{i=1}^k X_i$, where $X_i$'s are non-degenerate totally regular continua.
 Fix $\varrho>1$, $a_i\in X_i$ ($i=1,\dots,k$) and put $d_i=d_{X_i,a_i,a_i}$,
 $\CCc_i=\CCc_{X_i,a_i,a_i}$.
 Let $f_i:X_i\to X_{i+1}$ ($i=1,\dots,k-1$) and $f_k:X_k\to X_1$
 be LEL maps from Theorem~\ref{T:MAIN}.
 Finally, let $d$ be the metric on $X$ such that $d(x,y)=d_{i}(x,y)$ for any
 $x,y\in X_i$ ($i=1,\dots,k$) and $d(x,y)=2$ for $x\in X_i$, $y\in X_j$ ($i\ne j$).
 Since $d_{i}(X_i)\le 1$, the metric $d$ is compatible with the topology of $X$.

 Define $f:X\to X$ by $f|_{X_i}=f_i$ for $i=1,\dots,k$.
 For every $i$ the restriction $f^k|_{X_i}:X_i\to X_i$
 is LEL, hence it is exactly Devaney chaotic with positive finite entropy and specification
 by Proposition~\ref{P:tentLikeIsExact}. Since $f$ permutes
 $X_1,\dots,X_k$, the assertion follows.
\end{proof}

\medskip
\noindent
\textbf{Acknowledgment.} The author wishes to express his thanks
to \mL{}u\-bo\-m\'ir Snoha for his help with the preparation of the paper.
{The author was supported by the Slovak Research and Development Agency
under the contract No.~APVV-0134-10
and by the Slovak Grant Agency under the grants
VEGA~1/0855/08 and VEGA~1/0978/11.}



\begin{thebibliography}{99}

\bibitem{AC01}
 S. Agronsky and J. G. Ceder,
 \textit{Each Peano subspace of $E\sp k$ is an $\omega$-limit set},
 Real Anal. Exchange \textbf{17}~(1991/92), no.~1, 371--378.

\bibitem{AKLS}
 L. Alsed\`a, S. Kolyada, J. Llibre and \mL{}. Snoha,
 \textit{Entropy and periodic points for transitive maps},
 Trans. Amer. Math. Soc. \textbf{351}~(1999), no.~4, 1551--1573.


\bibitem{ARR}
 L. Alsed\`a, M. A. del R\'io and J. A. Rodr\'iguez,
 \textit{A splitting theorem for transitive maps},
 J. Math. Anal. Appl. \textbf{232}~(1999), no.~2, 359--375.

\bibitem{BingPartSet}
 R. H. Bing,
 \textit{Partitioning a set},
 Bull. Amer. Math. Soc.~\textbf{55}~(1949), 1101--1110.

\bibitem{Blo83}
 A. M. Blokh,
 \textit{Decomposition of dynamical systems on an interval},
 Russ. Math. Surv. \textbf{38}~(1983), 133--134.

\bibitem{BNT}
 R. D. Buskirk, J. Nikiel and E. D. Tymchatyn,
 \textit{Totally regular curves as inverse limits},
 Houston J. Math. \textbf{18}~(1992), no.~3, 319--327.

\bibitem{DGS}
 M. Denker, C. Grillenberger and K. Sigmund,
 \textit{Ergodic theory on compact spaces}, Springer-Verlag, Berlin, 1976.

\bibitem{Eil38}
 S. Eilenberg,
 \textit{On continua of finite length},
 Ann. Soc. Pol. Math.~\textbf{17}~(1938), 253--254.

\bibitem{Eil44}
 S. Eilenberg,
 \textit{Continua of finite linear measure II},
 Amer. J. Math.~\textbf{66}~(1944), 425--427.

\bibitem{EH}
 S. Eilenberg and O. G. Harrold,
 \textit{Continua of finite linear measure I},
 Amer. J. Math.~\textbf{65}~(1943), 137--146.

\bibitem{EngDim}
  R. Engelking,
  \textit{Dimension theory},
  North-Holland and PWN---Polish Scientific Publishers, Warsaw, 1978.

\bibitem{Fal}
 K. J. Falconer,
 \textit{The geometry of fractal sets},
 Cambridge University Press, Cambridge, 1986.

\bibitem{Fre92}
 D. H. Fremlin,
 \textit{Spaces of finite length},
 Proc. London Math. Soc.~(3) \textbf{64}~(1992), 449--486.

\bibitem{Fre94}
 D. H. Fremlin,
 \textit{Embedding spaces of finite length in $\RRR^3$},
 J. London Math. Soc.~(2) \textbf{49}~(1994), 150--162.

\bibitem{Har44}
 O. G. Harrold,
 \textit{The construction of a certain metric},
 Duke Math. J. \textbf{11}~(1944), 23--34.


\bibitem{Kur2}
 K. Kuratowski,
 \textit{Topology, vol. 2},
 Academic Press and PWN, Warszawa, 1968.


\bibitem{Macias}
  S. Mac\'ias,
  \emph{Topics on continua},
  Chapman \& Hall/CRC, Boca Raton, FL, 2005.

\bibitem{Mattila}
  P. Mattila,
  \emph{Geometry of sets and measures in Euclidean spaces. Fractals and rectifiability},
  Cambridge University Press, Cambridge, 1995.

\bibitem{Nad}
  S. B. Nadler,
  \textit{Continuum theory. An introduction},
  Monographs and Textbooks in Pure and Applied Mathematics, 158,
  Marcel Dekker, Inc., New York, 1992.

\bibitem{Nik}
  J. Nikiel,
  \textit{Locally connected curves viewed as inverse limits},
  Fund. Math. \textbf{133}~(1989), no.~2, 125--134.


\bibitem{Why35}
 G. T. Whyburn,
 \emph{Concerning continua of finite degree and local separating points},
 Amer. J. Math. \textbf{57}~(1935), no.~1, 11--16.

\end{thebibliography}
\end{document}